\documentclass[12pt, a4paper]{amsart}
\usepackage{amscd,amsmath,amssymb,amsthm,amsfonts}
\usepackage[alphabetic]{amsrefs}
\usepackage{mathrsfs}
\usepackage[shortlabels]{enumitem}
\usepackage[all]{xy}

\usepackage{xcolor}
\usepackage[hypertexnames=true,colorlinks = true, allcolors=blue,linktoc = all, pdffitwindow = false, urlbordercolor = white]{hyperref}%

\usepackage{tikz}
 
\usetikzlibrary{knots}
\usetikzlibrary{decorations.markings}

\usepackage{graphicx} 

\def\max{\operatorname{max}}

 \newcommand{\IN}[0]{\mathbb{N}}

 \newcommand{\IZ}[0]{\mathbb{Z}}

\newcommand{\CEE}[0]{\mathcal{E}_{\rm EVEN}}
 
 \renewcommand{\CD}[0]{\mathcal{D}}
\newcommand{\CE}[0]{\mathcal{E}} \newcommand{\CF}[0]{\mathcal{F}}

 \newcommand{\CT}[0]{\mathcal{T}}

\newcommand {\stab}{{\rm Stab}}

\newtheorem{theorem}{Theorem}[section]
\newtheorem*{theorem*}{Theorem}
\newtheorem*{proposition*}{Proposition}
\newtheorem{proposition}[theorem]{Proposition}
\newtheorem{lemma}[theorem]{Lemma}
\newtheorem*{lemma*}{Lemma}
\newtheorem{example}[theorem]{Example}

\newtheorem{definition}[theorem]{Definition}

\newtheorem{corollary}[theorem]{Corollary}
\newtheorem{remark}[theorem]{Remark}

\newtheorem{convention}{Convention}[section] 

\numberwithin{equation}{section}

\begin{document}
\title[The  planar $3$-colorable subgroup $\mathcal{E}$ of $F$ and its even part]
{The planar $3$-colorable subgroup $\mathcal{E}$ of Thompson's group $F$ and its even part}
\author{Valeriano Aiello} 
\address{Valeriano Aiello,
Dipartimento di Matematica, Universit\`a di Roma La Sapienza, P.le Aldo Moro
5, 00185 Roma, Italy}\email{valerianoaiello@gmail.com}
\author{Tatiana Nagnibeda} 
\address{Tatiana Nagnibeda,
Section de Math\'  ematiques, Universit\' e de Gen\` eve, Rue du Conseil-Gén\'eral, 7-9, 1205 Gen\`eve, Swtizerland
}\email{tatiana.smirnova-nagnibeda@unige.ch}

\begin{abstract}
We study the planar $3$-colorable subgroup $\mathcal{E}$ of Thompson's group $F$  and its even part $\CEE$.
The latter is obtained by cutting $\CE$ with a finite index subgroup of $F$ isomorphic to $F$, namely the 
rectangular subgroup $K_{(2,2)}$.  
We show that  
the even part $\CEE$ of the planar $3$-colorable subgroup admits a description in terms of stabilisers of suitable subsets of dyadic rationals.
As a consequence   $\CEE$ is closed in the sense of Golan and Sapir. 
We then study three quasi-regular representations associated with $\CEE$:
two are shown to be irreducible and one to be reducible.
\end{abstract}

\maketitle


\section*{Introduction}
In \cite{Jo14} and \cite{Jo16} Vaughan Jones developed a new method to construct unitary representations of   Thompson groups.
There are two families of representations,
one coming from 
planar algebras or tensor categories \cite{jo2, AJ} and 
another  coming from Pythagorean C$^*$-algebras 
  \cite{BJ, BP, AP, BrWi}. 
  There is also the representation   in \cite{BJ2}, defined through an isometry $R$ from a Hilbert space $H$ to $H\otimes H$.
Such  representations 
contain a vector called the vacuum vector $\Omega$ (that for the first family is canonical)  and 
 the corresponding 
coefficients   
$\langle \pi(g)\Omega,\Omega\rangle$
often
admit a nice 
interpretation in terms of knot and graph invariants, see \cite{Jo14, Ren, ACJ, ABC, 
 AiCo1, AiCo2} and also \cite{AB1, AB2, A2, GP}. 
In this paper we only consider the Thompson group $F$ and we think of its elements as equivalences classes of tree diagrams, \cite{CFP}.

Several interesting subgroups of $F$ arise as stabilisers of vacuum vectors in Jones' representations. 
For instance, the oriented subgroup $\vec{F}$ studied in \cite{GS, GS2, Ren}, 
the $3$-colorable subgroup $\CF$ \cite{Ren, TV2}  and the parabolic subgroups, i.e., stabilisers of points in $(0, 1)$, studied in \cite{Sav, Sav2, GS3, BJ, DF}.
This  paper
focuses on 
another subgroup $\CE$  that arises in the same way, see \cite{Jo19}, that we call 
the planar $3$-colorable subgroup.

 While analyzing $\CE$, one is 
 naturally brought to consider its even part $\CEE$,
 which is obtained by 
\emph{cutting} $\CE$ with the rectangular subgroup $K_{(2,2)}
:=\{f\in F\; | \; \log_2 f'(0)\in 2 \IZ, \log_2 f'(1)\in 2\IZ 		\}
$. 
In fact, for each pair of integers $(m,n)$ there is a rectangular subgroup $K_{(m,n)}$, consisting of those elements of $f\in F$ for which the logarithms (in base $2$) of the derivates at $0$ and $1$ are in $m\IZ$ and $n\IZ$, respectively.  As will be shown in Lemma \ref{lemmarootcol}, for the elements of $\CE$ the integers $m$ and $n$ must have the same parity and, therefore, the most natural choice is to \emph{cut} $\CE$ with $K_{(2,2)}$. However, it would be interesting to study the subgroups obtained by intersecting the planar $3$-colorable subgroup with all the other rectangular subgroups.

For some countable families of Jones' representations the coefficient  of the vacuum vector $\langle \pi(g)\Omega, \Omega\rangle$ has been realized (up to a normalisation) as an evaluation of the chromatic polynomial of a certain graph $\Gamma(g)$ associated with the equivalence class of the tree diagram $g\in F$.
The graph construction is specific for each family of representations and we explain it below.
In each family, the representations that correspond to evaluations at integers are of special interest. In particular, the subgroups
$\CF$ and $\CE$ are the stabilisers of the vacuum vector in the unique representations of the respective families
which correspond to an evaluation at a rational number (that happens to be equal to $3$). These two representations are also realized as quasi-regular representations of $F$ associated with $\CF$ and $\CE$ and they have been shown to be irreducible in \cite{Jo19, TV2}.
In the case of $\CF$, the graph $\Gamma(g)$ is constructed as follows. We take a tree diagram and we \emph{close} it as shown in Figure \ref{differentcoefficients}. Then $\Gamma(g)$ is the dual graph of this closure of the tree diagram. An element of $g$ in $F$ is in $\CF$ precisely when $\Gamma(g)$ is $3$-colorable.
In the other family of representations, $\Gamma(g)$ is just the dual graph of the tree diagram.
There is also a third graph $\Gamma(g)$ that appears in the description of the vacuum coefficient of a family of Jones' representations of $F$, see \cite[Definition 4.1.2]{Jo14}. 
In this case the vacuum coefficient is (up to a normalisation) also an evaluation of the chromatic polynomial Chr$_{\Gamma(g)}(Q)$. The stabiliser of the vacuum vector in the representation corresponding to
 $Q=2$ yields the oriented subgroup $\vec{F}$, which has been studied in \cite{GS, GS2, A}. See also \cite{TV} for a study of its ternary version $\vec{F}_3$.

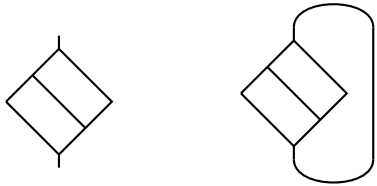
\begin{figure}
\phantom{This text will be invisible} 
\[
\begin{tikzpicture}[x=.35cm, y=.35cm,
    every edge/.style={
        draw,
      postaction={decorate,
                    decoration={markings}
                   }
        }
]

\node at (-1.25,-3) {\;};

\draw[thick] (0,0) -- (2,2)--(4,0)--(2,-2)--(0,0);
 \draw[thick] (1,1) -- (2,0)--(3,-1);
 \draw[thick] (2,2) -- (2,2.5);
 \draw[thick] (2,-2) -- (2,-2.5);
 
\end{tikzpicture}\qquad
\;\;
\begin{tikzpicture}[x=.35cm, y=.35cm,
    every edge/.style={
        draw,
      postaction={decorate,
                    decoration={markings}
                   }
        }
]

\node at (-1.25,-3) {\;};

\draw[thick] (0,0) -- (2,2)--(4,0)--(2,-2)--(0,0);
 \draw[thick] (1,1) -- (2,0)--(3,-1);

 \draw[thick] (2,2)--(2,2.5);

 \draw[thick] (2,-2)--(2,-2.5);

\draw[thick] (2,2.5) to[out=90,in=90] (5,2.5);  
\draw[thick] (2,-2.5) to[out=-90,in=-90] (5,-2.5);  
 \draw[thick] (5,-2.5)--(5,2.5);

\end{tikzpicture} 
\]
\caption{A tree diagram and its closure.}\label{differentcoefficients}
\end{figure}

The    paper is structured as follows.
After some preliminaries in
 Section \ref{sec1},  
in Section \ref{sec2} we define the subgroups $\CE$ and $\CEE$
and determine their generators in Theorems \ref{genE} and \ref{genE1}.
We observe that $\CE$ is not a maximal subgroup in $F$. 
  Then, we show that $\CE$ and $\CF$ together generate a finite index subgroup of $F$. 
We also  investigate how $\CE$ and $\CEE$ act on the set of dyadic rationals in Section \ref{sec3}. 
In Theorem \ref{lemma-inclu}, we describe $\CEE$ 
as the intersection of stabilizers of certain subsets of dyadic rationals in $(0,1)$.
This implies that
  $\CEE$ is  closed in the sense of \cite{GS2},
see
Corollary \ref{CEEclosed}.

In contrast with $\vec{F}$, $\CF$ which are both closed and self-commensurating, the subgroup $\CE$ turns out not to be closed.
However it follows from a result of Jones  
 \cite{Jo19} that $\CE$ coincides with its commensurator.
 As for $\CEE$, 
 the main result of the paper is Theorem \ref{theo2}, where we prove that its commensurator in $F$ is $\CE$ and that it self-commensurating in
 $K_{(2,1)}$ and $K_{(2,2)}$. In particular, since 
$K_{(2,1)}$ and $K_{(2,2)}$ are isomorphic with $F$ via explicit isomorphisms $\alpha: F\to K_{(2,1)}$, $\theta: F\to K_{(2,2)}$ (defined in \cite{TV2} and also described in Section \ref{sec3}), 
the quasi-regular representations of $F$ associated with $\alpha^{-1}(\CEE)$ and $\theta^{-1}(\CEE)$ are irreducible. 
 
\section{Preliminaries and notation}\label{sec1}
In this section we recall the definitions of Thompson's group $F$, of the Brown-Thompson groups $F_k$ and of the $3$-colorable subgroup $\CF$.  
The interested reader is referred to \cite{CFP} and \cite{B} for more information on $F$, to \cite{Brown} for $F_k$, to \cite{Ren} and \cite{TV2} for $\CF$.

Thompson's group $F$ is the group of all piecewise linear homeomorphisms of the unit interval $[0,1]$ that are differentiable everywhere except at finitely many dyadic rationals numbers and such that on the intervals of differentiability the derivatives are powers of $2$. We adopt the standard notation: $f\cdot g(t)=g(f(t))$.

Thompson's group   has the following infinite presentation
$$
F=\langle x_0, x_1, x_2, \ldots \; | \; x_nx_k=x_kx_{n+1} \quad \forall \; k<n\rangle\, .
$$
The monoid generated by $x_0, x_1, x_2, \ldots$ is denoted by $F_+$. Its elements are said to be positive. Note that $x_0$ and $x_1$ are enough to generate $F$.

Every element $g$ of $F$ can be written in a unique way as 
$$
x_0^{a_0}\cdots x_n^{a_n}x_0^{-b_0}\cdots x_0^{-b_0}
$$
where $a_0$, \ldots , $a_n$, $b_0$, \ldots , $b_n\in \IN_0$,
exactly one between $a_n$ and $b_n$ is non-zero, 
and if $a_i\neq 0$ and $b_i\neq 0$, then $a_{n+1}\neq 0$ or $b_{n+1}\neq 0$ for all $i$. This is the normal form of $g$.

\begin{figure}
\phantom{This text will be invisible} 
\[
\begin{tikzpicture}[x=.35cm, y=.35cm,
    every edge/.style={
        draw,
      postaction={decorate,
                    decoration={markings}
                   }
        }
]

\node at (-1.5,0) {$\scalebox{1}{$x_0=$}$};
\node at (-1.25,-3) {\;};

\draw[thick] (0,0) -- (2,2)--(4,0)--(2,-2)--(0,0);
 \draw[thick] (1,1) -- (2,0)--(3,-1);

 \draw[thick] (2,2)--(2,2.5);

 \draw[thick] (2,-2)--(2,-2.5);

\end{tikzpicture}\qquad
\;\;
\begin{tikzpicture}[x=.35cm, y=.35cm,
    every edge/.style={
        draw,
      postaction={decorate,
                    decoration={markings}
                   }
        }
]

\node at (-3.5,0) {$\scalebox{1}{$x_1=$}$};
\node at (-1.25,-3.25) {\;};

\draw[thick] (2,2)--(1,3)--(-2,0)--(1,-3)--(2,-2);

\draw[thick] (0,0) -- (2,2)--(4,0)--(2,-2)--(0,0);
 \draw[thick] (1,1) -- (2,0)--(3,-1);

 \draw[thick] (1,3)--(1,3.5);
 \draw[thick] (1,-3)--(1,-3.5);

\end{tikzpicture}
\]
\caption{The generators of $F=F_2$.}\label{genThompsonF}
\end{figure}
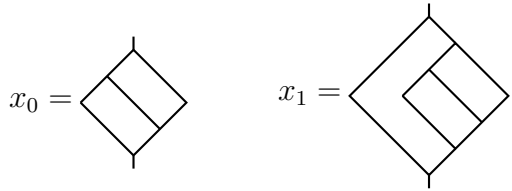
 The projection of $F$ onto its abelianisation is denoted by $\pi: F\to F/[F,F]=\IZ\oplus \IZ$ and it admits a nice interpretation when $F$ is seen as a group of homeomorphisms: $\pi(f)=(\log_2 f'(0),\log_2 f'(1))$. 
 
 There is a family of groups generalizing the Thompson group: the Brown-Thompson groups.
For any $k\geq 2$, the Brown-Thompson group $F_k$ may be defined by the following presentation  
$$
\langle y_0, y_1, \ldots \; | \; y_ny_l=y_ly_{n+k-1} \quad \forall \; l<n\rangle\, .
$$
The elements $y_0, y_1, \ldots , y_{k-1}$ generate $F_k$.

 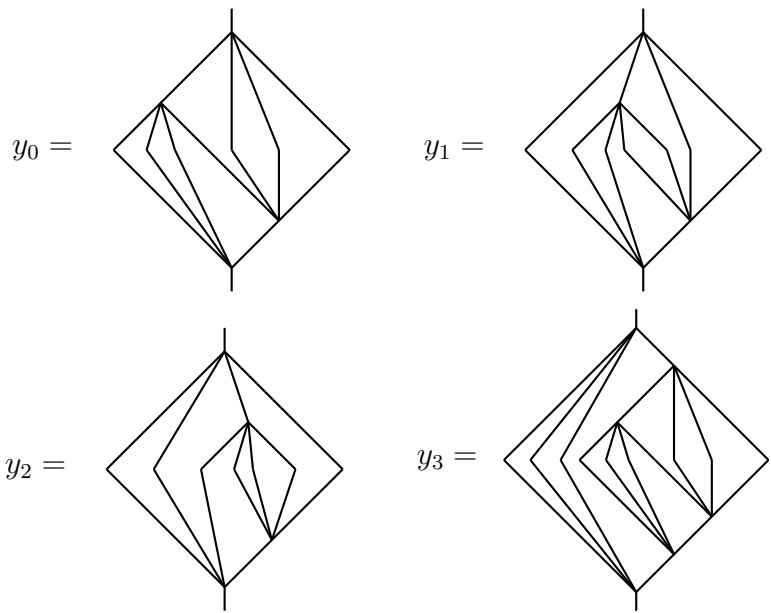
\begin{figure}
\phantom{This text will be invisible} 
\[
\begin{tikzpicture}[x=1.25cm, y=1.25cm,
    every edge/.style={
        draw,
      postaction={decorate,
                    decoration={markings}
                   }
        }
]

\draw[thick] (0,0)--(.5,.5)--(1,0);
\draw[thick] (0.5,0.5)--(.35,0);
\draw[thick] (0.5,0.5)--(.65,0);
\node at (0,-1.2) {$\;$};
\node at (-.75,0) {$\scalebox{1}{$y_0=$}$};

\draw[thick] (1.25,1.5)--(1.25,1.25);
\draw[thick] (1.25,-1.5)--(1.25,-1.25);

\draw[thick] (0.5,0.5)--(1.25,1.25);
\draw[thick] (1.25,0)--(1.25,1.25);
\draw[thick] (1.75,0)--(1.25,1.25);
\draw[thick] (2.5,0)--(1.25,1.25);

\draw[thick] (2.5,0)--(1.25,-1.25)--(0,0);
\draw[thick] (1.25,-1.25)--(0.35,0);
\draw[thick] (1.25,-1.25)--(0.65,0);

\draw[thick] (1.75,-.75)--(1,0);
\draw[thick] (1.75,-.75)--(1.25,0);
\draw[thick] (1.75,-.75)--(1.75,0);

\end{tikzpicture}
\qquad
\begin{tikzpicture}[x=1.25cm, y=1.25cm,
    every edge/.style={
        draw,
      postaction={decorate,
                    decoration={markings}
                   }
        }
]

\draw[thick] (0.5,0)--(1,.5)--(1.5,0);
\draw[thick] (1,0.5)--(.85,0);
\draw[thick] (1,0.5)--(1.05,0);
\node at (0,-1.2) {$\;$};
\node at (-.75,0) {$\scalebox{1}{$y_1=$}$};

\draw[thick] (0,0)--(1.25,1.25);
\draw[thick] (1,.5)--(1.25,1.25);
\draw[thick] (1.75,0)--(1.25,1.25);
\draw[thick] (2.5,0)--(1.25,1.25);

\draw[thick] (2.5,0)--(1.25,-1.25)--(0,0);
\draw[thick] (1.25,-1.25)--(0.5,0);
\draw[thick] (1.25,-1.25)--(0.85,0);

\draw[thick] (1.75,-.75)--(1.05,0);
\draw[thick] (1.75,-.75)--(1.5,0);
\draw[thick] (1.75,-.75)--(1.75,0);

\draw[thick] (1.25,1.5)--(1.25,1.25);
\draw[thick] (1.25,-1.5)--(1.25,-1.25);

\end{tikzpicture}
\]
\[
\begin{tikzpicture}[x=1.25cm, y=1.25cm,
    every edge/.style={
        draw,
      postaction={decorate,
                    decoration={markings}
                   }
        }
]

\draw[thick] (1,0)--(1.5,.5)--(2,0);
\draw[thick] (1.5,0.5)--(1.35,0);
\draw[thick] (1.5,0.5)--(1.55,0);
\node at (0,-1.2) {$\;$};
\node at (-.75,0) {$\scalebox{1}{$y_2=$}$};

\draw[thick] (0,0)--(1.25,1.25);
\draw[thick] (.5,0)--(1.25,1.25);
\draw[thick] (1.5,.5)--(1.25,1.25);
\draw[thick] (2.5,0)--(1.25,1.25);

\draw[thick] (2.5,0)--(1.25,-1.25)--(0,0);
\draw[thick] (1.25,-1.25)--(0.5,0);
\draw[thick] (1.25,-1.25)--(1,0);

\draw[thick] (1.75,-.75)--(1.35,0);
\draw[thick] (1.75,-.75)--(1.55,0);
\draw[thick] (1.75,-.75)--(2,0);

\node at (1.75,.75) {$\;$};

\draw[thick] (1.25,1.5)--(1.25,1.25);
\draw[thick] (1.25,-1.5)--(1.25,-1.25);

\end{tikzpicture}
\qquad
\begin{tikzpicture}[x=1cm, y=1cm,
    every edge/.style={
        draw,
      postaction={decorate,
                    decoration={markings}
                   }
        }
]

\draw[thick] (.75,1.75)--(.75,2);
\draw[thick] (.75,-1.75)--(.75,-2);

\draw[thick] (0,0)--(.5,.5)--(1,0);
\draw[thick] (0.5,0.5)--(.35,0);
\draw[thick] (0.5,0.5)--(.65,0);
\node at (1.75,-.75) {$\;$};
\node at (-1.75,0) {$\scalebox{1}{$y_3=$}$};

\draw[thick] (0.5,0.5)--(1.25,1.25);
\draw[thick] (1.25,0)--(1.25,1.25);
\draw[thick] (1.75,0)--(1.25,1.25);
\draw[thick] (2.5,0)--(1.25,1.25);

\draw[thick] (-1,0)--(.75,1.75)--(1.25,1.25);
\draw[thick] (-1,0)--(.75,-1.75)--(1.25,-1.25);
\draw[thick] (-.65,0)--(0.75,1.75);
\draw[thick] (-.25,0)--(0.75,1.75);
\draw[thick] (-.65,0)--(0.75,-1.75);
\draw[thick] (-.25,0)--(0.75,-1.75);

\draw[thick] (2.5,0)--(1.25,-1.25)--(0,0);
\draw[thick] (1.25,-1.25)--(0.35,0);
\draw[thick] (1.25,-1.25)--(0.65,0);

\draw[thick] (1.75,-.75)--(1,0);
\draw[thick] (1.75,-.75)--(1.25,0);
\draw[thick] (1.75,-.75)--(1.75,0);

\end{tikzpicture}
\]
\caption{The generators of the Brown-Thompson group $F_4$.}\label{genThompsonF4}
\end{figure}

 Going back to $F$,
there is still another description  which is relevant to this paper: the elements of $F$ can be seen as pairs $(T_+,T_-)$ of planar binary rooted trees (with the same number of leaves).  We draw one tree upside down on top of the other; $T_+$ is  the top tree, while $T_-$ is the bottom tree.
Any pair of trees $(T_+,T_-)$ represented in this way is called a tree diagram.  Two pairs of trees are said to be equivalent if they differ by pairs of opposing carets, namely
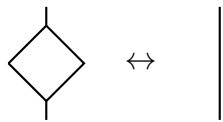
\[\begin{tikzpicture}[x=.5cm, y=.5cm,
    every edge/.style={
        draw,
      postaction={decorate,
                    decoration={markings}
                   }
        }
]

 \draw[thick] (0,0)--(1,1)--(2,0)--(1,-1)--(0,0); 
 \draw[thick] (1,1.5)--(1,1); 
 \draw[thick] (1,-1.5)--(1,-1); 
\node at (0,-1.2) {$\;$};
\node at (3.5,0) {$\scalebox{1}{$\leftrightarrow$}$};

\end{tikzpicture}
\begin{tikzpicture}[x=.5cm, y=.5cm,
    every edge/.style={
        draw,
      postaction={decorate,
                    decoration={markings}
                   }
        }
]

  \draw[thick] (1,1.5)--(1,-1.5); 
\node at (0,-1.2) {$\;$};
 
\end{tikzpicture}
\]
Every equivalence class of pairs of trees (i.e. an element of $F$) gives rise to exactly one tree diagram which is reduced, in the sense that 
the number of its vertices is minimal, \cite{B}.  See Figure \ref{genThompsonF} for their description in terms of pairs of binary trees.
Similarly, the elements of $F_k$ are described by pairs $k$-ary trees, see e.g. the generators of $F_4$ that are displayed in Figure \ref{genThompsonF4}.

\begin{convention}\label{conventiondrawings}
We make a convention about how we draw trees on the plane. The roots of our planar binary trees are drawn as vertices of degree $3$ and hence each tree diagram has the uppermost and lowermost vertices of degree  $1$, which lie respectively on the lines   $y=1$ and $y=-1$.
 The leaves of the trees sit on the $x$-axis, precisely on the non-negative integers. 
\end{convention}

Any tree diagram partitions the strip bounded by the lines $y=1$ and $y=-1$ in regions.
This strip may or may not be $3$-colorable, i.e., it may or may not be 
 possible to assign the colors $\IZ_3=\{0,1, 2\}$ to the regions of the strip in such a way that if two regions share an edge, they   have different colors. 
  By convention, we assign the following colours to the regions near the roots 
  \begin{eqnarray}\label{convention-colors}
\begin{tikzpicture}[x=.5cm, y=.5cm,
    every edge/.style={
        draw,
      postaction={decorate,
                    decoration={markings}
                   }
        }
] 

 \draw[thick] (0,0) -- (1,1)--(2,0);
 \draw[thick] (1,1) -- (1,2);



\node at (0,1) {$0$};
\node at (2,1) {$1$};
\node at (1,0.2) {$2$};

 \draw[thick] (-1,2) -- (3,2);


\end{tikzpicture}
\qquad\qquad 
\begin{tikzpicture}[x=.5cm, y=.5cm,
    every edge/.style={
        draw,
      postaction={decorate,
                    decoration={markings}
                   }
        }
] 

 \draw[thick] (-1,0) -- (3,0);

 \draw[thick] (0,2) -- (1,1)--(2,2);
 \draw[thick] (1,1) -- (1,0);



\node at (0,1) {$0$};
\node at (2,1) {$1$};
\node at (1,1.8) {$2$};

\end{tikzpicture}
\end{eqnarray}
Once we make this convention, if the strip is $3$-colorable, 
there exists a unique colouring. 
The \textbf{$3$-colorable subgroup} $\CF$  consists of the elements of $F$ for which the corresponding  strip is $3$-colorable. 
For example, this is the strip corresponding to $x_0$ (which is not $3$-colorable)
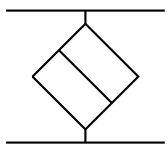
\[
\begin{tikzpicture}[x=.35cm, y=.35cm,
    every edge/.style={
        draw,
      postaction={decorate,
                    decoration={markings}
                   }
        }
]

 \draw[thick] (-1,2.5)--(5,2.5);
 \draw[thick] (-1,-2.5)--(5,-2.5);

 \node at (-1.25,-3) {\;};

\draw[thick] (0,0) -- (2,2)--(4,0)--(2,-2)--(0,0);
 \draw[thick] (1,1) -- (2,0)--(3,-1);

 \draw[thick] (2,2)--(2,2.5);

 \draw[thick] (2,-2)--(2,-2.5);

\end{tikzpicture}
\]
 The subgroup  $\CF$
 was shown to be
  isomorphic to the Brown-Thompson group $F_4$ by Ren in \cite{Ren}
  by means of the isomorphism 
  $\gamma: F_4\to \CF$
  obtained by
   replacing every $5$-valent vertex of $4$-ary trees by the complete binary tree with $4$ leaves 
(in the sequel we refer to this tree as the \emph{basic tree}, see Figure \ref{fig-ren-map-2}), 
\cite{BCS}.
The images of   $y_0$, $y_1$, $y_2$, $y_3$ yield the following elements 
$w_0:=x_0^2x_1x_2^{-1}$,
$w_1:=x_0x_1^2x_0^{-1}$,
$w_2:=x_1^2x_3x_2^{-1}$,
$w_3:=x_2^2x_3x_4^{-1}$ (see Figure \ref{genCF}). 
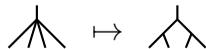
\begin{figure}
\[\begin{tikzpicture}[x=.75cm, y=.75cm,
    every edge/.style={
        draw,
      postaction={decorate,
                    decoration={markings}
                   }
        }
]

\draw[thick] (0,0)--(.5,.5)--(1,0);
\draw[thick] (0.5,0.5)--(.5,.75);
\draw[thick] (0.5,0.5)--(.35,0);
\draw[thick] (0.5,0.5)--(.65,0);
\node at (0,-1.2) {$\;$};
\node at (1.75,0.25) {$\scalebox{1}{$\mapsto$}$};

\end{tikzpicture}
\begin{tikzpicture}[x=.75cm, y=.75cm,
    every edge/.style={
        draw,
      postaction={decorate,
                    decoration={markings}
                   }
        }
]

\draw[thick] (0.5,0.75)--(.5,.5);
\draw[thick] (0.65,0)--(.75,.25);
\draw[thick] (0.35,0)--(.25,.25);
\draw[thick] (0,0)--(.5,.5)--(1,0);
\node at (0,-1.2) {$\;$};
\end{tikzpicture}
\]
\caption{Ren's map  $\Phi$ 
 from the set of $4$-ary trees to binary trees.  }
  \label{fig-ren-map-2}
\end{figure} 
This subgroup was also studied by the authors in \cite{TV2} in relation to maximal subgroups of infinite index of $F$.   \begin{figure}
\phantom{This text will be invisible} 
\[
\begin{tikzpicture}[x=.75cm, y=.75cm,
    every edge/.style={
        draw,
      postaction={decorate,
                    decoration={markings}
                   }
        }
]

\node at (-1.5,-.25) {$\scalebox{1}{$w_0=$}$};

\draw[thick] (0,0)--(2,2)--(4,0);
\draw[thick] (0,0)--(2,-2)--(4,0);
\draw[thick] (1,0)--(1.5,0.5)--(2,0);
\draw[thick] (1.5,0.5)--(1,1);
\draw[thick] (3,0)--(1.5,1.5);
\draw[thick] (3,0)--(2.5,-.5); 
\draw[thick] (4,0)--(2,2);
\draw[thick] (4,0)--(3.5,-.5);
\draw[thick] (1,0)--(2.5,-1.5);
\draw[thick] (2,0)--(3,-1);

\draw[thick] (2,2)--(2,2.25);
\draw[thick] (2,-2)--(2,-2.25);

\node at (0,-1.2) {$\;$};
\end{tikzpicture}
\qquad 
\begin{tikzpicture}[x=.75cm, y=.75cm,
    every edge/.style={
        draw,
      postaction={decorate,
                    decoration={markings}
                   }
        }
]

\node at (-1.5,-.25) {$\scalebox{1}{$w_1=$}$};

\draw[thick] (0,0)--(2,2)--(4,0);
\draw[thick] (0,0)--(2,-2)--(4,0);

\draw[thick] (1,0)--(1.5,0.5)--(2,0);
\draw[thick] (2,0)--(3,-1);
\draw[thick] (1,0)--(.5,-.5);
\draw[thick] (1.5,.5)--(2,1)--(3,0);
\draw[thick] (3,0)--(3.5,-.5);
\draw[thick] (2,1)--(1.5,1.5);

\draw[thick] (2,2)--(2,2.25);
\draw[thick] (2,-2)--(2,-2.25);

\node at (0,-1.2) {$\;$};
\end{tikzpicture}
\]

\[
\begin{tikzpicture}[x=.75cm, y=.75cm,
    every edge/.style={
        draw,
      postaction={decorate,
                    decoration={markings}
                   }
        }
]

\node at (-2.5,-.25) {$\scalebox{1}{$w_2=$}$};
 
\draw[thick] (-1,0)--(1.5,2.5)--(2,2);
\draw[thick] (-1,0)--(1.5,-2.5)--(2,-2);

\draw[thick] (0,0)--(2,2)--(4,0);
\draw[thick] (0,0)--(2,-2)--(4,0);

\draw[thick] (2,0)--(2.5,.5)--(3,0);

\draw[thick] (1,0)--(.5,.5);
\draw[thick] (2.5,.5)--(1.5,1.5);

\draw[thick] (3,0)--(3.5,-.5);
\draw[thick] (2,0)--(1.5,-.5)--(1,0);
\draw[thick] (1.5,-.5)--(2.5,-1.5);

\draw[thick] (1.5,2.5)--(1.5,2.75);
\draw[thick] (1.5,-2.5)--(1.5,-2.75);

\node at (0,-1.2) {$\;$};
\end{tikzpicture}
\qquad 
\begin{tikzpicture}[x=.75cm, y=.75cm,
    every edge/.style={
        draw,
      postaction={decorate,
                    decoration={markings}
                   }
        }
]

\node at (-3.5,-.25) {$\scalebox{1}{$w_3=$}$};
 
\draw[thick] (-1,0)--(1.5,2.5)--(2,2);
\draw[thick] (-1,0)--(1.5,-2.5)--(2,-2);

\draw[thick] (-2,0)--(1,3)--(1.5,2.5);
\draw[thick] (-2,0)--(1,-3)--(1.5,-2.5);

\draw[thick] (0,0)--(2,2)--(4,0);
\draw[thick] (0,0)--(2,-2)--(4,0);
\draw[thick] (1,0)--(1.5,0.5)--(2,0);
\draw[thick] (1.5,0.5)--(1,1);
\draw[thick] (3,0)--(1.5,1.5);
\draw[thick] (3,0)--(2.5,-.5); 
\draw[thick] (4,0)--(2,2);
\draw[thick] (4,0)--(3.5,-.5);
\draw[thick] (1,0)--(2.5,-1.5);
\draw[thick] (2,0)--(3,-1);

\draw[thick] (1,3)--(1,3.25);
\draw[thick] (1,-3)--(1,-3.25);

\node at (0,-1.2) {$\;$};
\end{tikzpicture}
\]
 \caption{The generators of $\CF$. 
 }\label{genCF}
\end{figure}
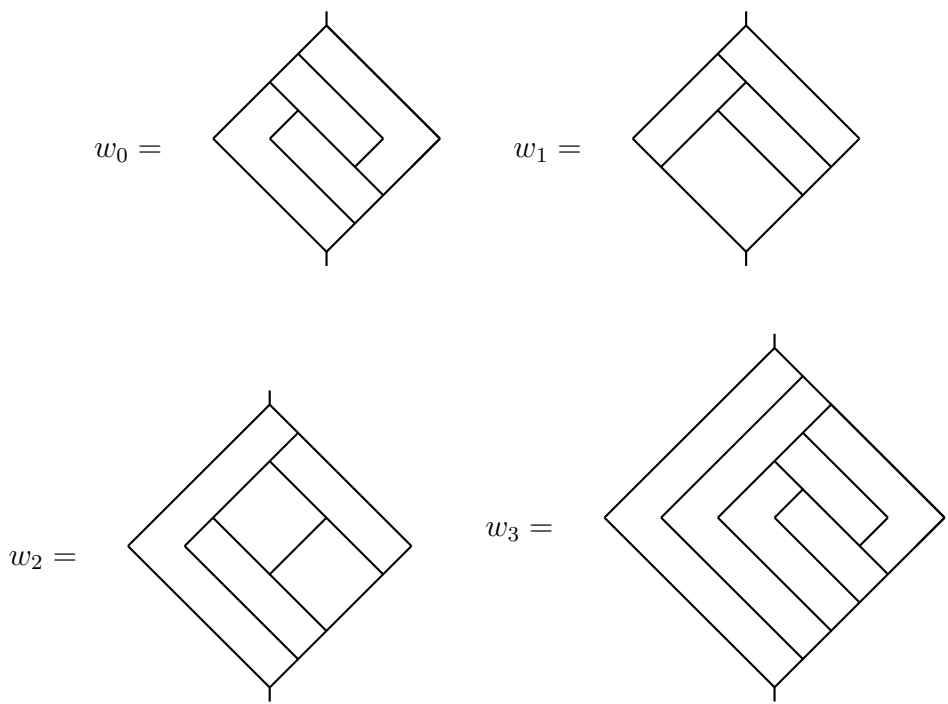

An important role in this article is played by the rectangular subgroups of $F$, which were introduced in \cite{BW} as
\begin{align*}
K_{(a,b)}&:=\{f\in F\; | \; \log_2f'(0)\in a\IZ, \log_2f'(1)\in b\mathbb{Z}\} \qquad a, b\in\IN
\end{align*}
These subgroups can be characterised as the only finite index subgroups of $F$ isomorphic with $F$ \cite[Theorem 1.1]{BW}.
 
Denote by $\{0, 1\}^*$ the set of finite binary words, i.e.,  finite sequences of $0$ and $1$.
There exists a  map $\rho$ between  
finite binary words 
and
the dyadic rationals in the open unit interval $\CD:=\IZ[1/2]\cap (0,1)$ given by the formula $\rho(a_1\ldots a_n):=\sum_{i=1}^n a_i 2^{-i}$ which is bijective when  restricted to finite words ending with $1$ (i.e. $a_n=1$).  
 
Recall that  the group $F$ acts  on $[0,1]$. 
Given a number $t\in [0,1]$ expressed in its binary expansion, it enters into the top of a tree diagram, follows a path towards the root of the bottom tree according to the rules portrayed in 
Figure \ref{compute}, and   emerges at the bottom as the image of $t$ under the element of $F$ represented by the tree diagram, \cite{BM}. 
Note that there is a change of direction only when the number comes across a vertex of degree $3$ (i.e., the number is unchanged when it comes across a leaf). 

\begin{figure}
\phantom{This text will be invisible} 
 \[
 \begin{tikzpicture}[x=1cm, y=1cm,
    every edge/.style={
        draw,
      postaction={decorate,
                    decoration={markings}
                   }
        }
]

\draw[thick] (0,0) -- (.5,1)--(1,0);
 \draw[thick] (.5,1)--(.5,1.5);

\fill (0,0)  circle[radius=.75pt];
 \fill (1,0)  circle[radius=.75pt];
\fill (.5,1.5)  circle[radius=.75pt];

\node at (.5,1.85) {$\scalebox{1}{$.0\alpha$}$};
\node at (0,-.25) {$\scalebox{1}{$.\alpha$}$};

\end{tikzpicture}
\qquad\qquad
\begin{tikzpicture}[x=1cm, y=1cm,
    every edge/.style={
        draw,
      postaction={decorate,
                    decoration={markings}
                   }
        }
]

\draw[thick] (0,0) -- (.5,1)--(1,0);
 \draw[thick] (.5,1)--(.5,1.5);

\fill (0,0)  circle[radius=.75pt];
 \fill (1,0)  circle[radius=.75pt];
\fill (.5,1.5)  circle[radius=.75pt];

\node at (.5,1.85) {$\scalebox{1}{$.1\alpha$}$};
\node at (1,-.25) {$\scalebox{1}{$.\alpha$}$};

\end{tikzpicture}
\qquad
\qquad
\begin{tikzpicture}[x=1cm, y=1cm,
    every edge/.style={
        draw,
      postaction={decorate,
                    decoration={markings}
                   }
        }
]

\draw[thick] (0,1.5) -- (.5,.5)--(1,1.5);
 \draw[thick] (.5,0)--(.5,.5);

\fill (0,1.5)  circle[radius=.75pt];
\fill (.5,0)  circle[radius=.75pt];
\fill (1,1.5)  circle[radius=.75pt];
 
\node at (.5,-.25) {$\scalebox{1}{$.0\alpha$}$};
\node at (0,1.75) {$\scalebox{1}{$.\alpha$}$};
\node at (1,-.25) {$\scalebox{1}{}$};

\end{tikzpicture}
\qquad\qquad
\begin{tikzpicture}[x=1cm, y=1cm,
    every edge/.style={
        draw,
      postaction={decorate,
                    decoration={markings}
                   }
        }
]

\draw[thick] (0,1.5) -- (.5,.5)--(1,1.5);
 \draw[thick] (.5,0)--(.5,.5);

\fill (0,1.5)  circle[radius=.75pt];
\fill (.5,0)  circle[radius=.75pt];
\fill (1,1.5)  circle[radius=.75pt];
 
\node at (.5,-.25) {$\scalebox{1}{$.1\alpha$}$};
\node at (1,1.75) {$\scalebox{1}{$.\alpha$}$};
\node at (1,-.25) {$\scalebox{1}{}$};

\end{tikzpicture}
\]
\caption{The local rules for computing the action of $F$ on numbers expressed in binary expansion.}\label{compute}
\end{figure}
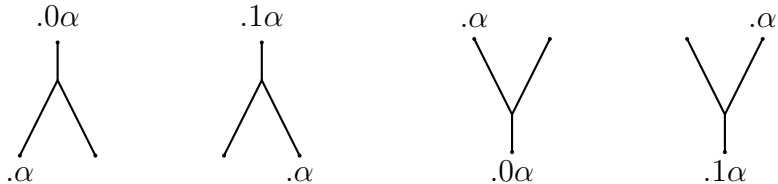

 \section{The planar $3$-colorable subgroup 
 and its even part} \label{sec2}
In this section    
 we introduce the planar $3$-colorable subgroup $\CE$ and its even part $\CEE$ and find  generating sets for them.

 As recalled above,
  the elements of $F$ 
 corresponding to tree diagrams
  for which the strip bounded by the lines $y=\pm 1$ is $3$-colorable form a subgroup called the $3$-colorable subgroup $\CF$.
  Observe that if we ignore the lines $y=\pm 1$, each 
  tree diagram
   partitions the plane into regions.
 The object of study of this article are the elements of $F$  for which this plane partition 
 is $3$-colorable. They form a subgroup called the \textbf{planar $3$-colorable subgroup} $\CE$.
 It is clear that  the plane partition
 being  $3$-colorable does not depend on equivalence class of the tree diagrams since 
 a $3$-coloring naturally extends when 
 a pair of opposing carets is added.
  When we talk about $\CE$ we remove the upward edge sprouting from the vertex of degree $1$  of the top tree and the downward edge sprouting from the root of the bottom tree, since this does not affect the fact that the partition is $3$-colorable.

   \begin{convention}\label{conventionCE}
    If the plane partition
  is $3$-colorable, then there are exactly six  colorings
   and if we assign the following colors to the regions near the root of the top tree
    \begin{eqnarray}\label{convention-colors}
\begin{tikzpicture}[x=.5cm, y=.5cm,
    every edge/.style={
        draw,
      postaction={decorate,
                    decoration={markings}
                   }
        }
] 

 \draw[thick] (0,0) -- (1,1)--(2,0);



\node at (1,2) {$0$};
\node at (1,0.2) {$1$};

\end{tikzpicture}
\end{eqnarray}  
then 
there exists a unique colouring. 
  \end{convention}

\begin{example}\label{esex0}
$x_0\in \CE$: 
\[
\begin{tikzpicture}[x=.35cm, y=.35cm,
    every edge/.style={
        draw,
      postaction={decorate,
                    decoration={markings}
                   }
        }
]

 \node at (-1.25,-3) {\;};

\draw[thick] (0,0) -- (2,2)--(4,0)--(2,-2)--(0,0);
 \draw[thick] (1,1) -- (2,0)--(3,-1);
%
%

\node at (-1,0) {$0$};
\node at (1,0) {$2$};
\node at (3,0) {$1$};

\end{tikzpicture}
\]
\end{example}  
This example shows that, unlike in the case of the $3$-colorable subgroup  \cite{TV2}, the colorings of the regions near the root of the bottom tree are not necessarily the same as those 
near the root of the top tree.
In the following lemma we show that the colors of the regions  near root of the bottom tree depend on the derivative of the homeomorphism  at the endpoints of the unit interval.
\begin{lemma}\label{lemmarootcol}
Let $f\in\CE$. 
Then the following conditions are equivalent
\begin{enumerate}
\item  $\log_2 f'(0)\in 2\IZ$ (respectively $\log_2 f'(0)\in 2\IZ+1$); 
\item $\log_2 f'(1)\in 2\IZ$ (respectively $\log_2 f'(1)\in 2\IZ+1$);
\item  the colors near the root of the bottom tree are 
\[
\begin{tikzpicture}[x=.5cm, y=.5cm,
    every edge/.style={
        draw,
      postaction={decorate,
                    decoration={markings}
                   }
        }
] 

 \draw[thick] (0,2) -- (1,1)--(2,2);
%


\node at (1,0) {$0$};
\node at (1,1.8) {$1$};
\node at (1,-1.5) {$\,$};

\end{tikzpicture}
\qquad 
\qquad  
\begin{tikzpicture}[x=.5cm, y=.5cm,
    every edge/.style={
        draw,
      postaction={decorate,
                    decoration={markings}
                   }
        }
] 

 \draw[thick] (0,2) -- (1,1)--(2,2);
%


\node at (1,0) {$0$};
\node at (1,1.8) {$2$};

\draw[thick] (-5,-1.5) to[out=125,in=-125] (-5,2.5);  
\draw[thick] (2.5,-1.5) to[out=55,in=-55] (2.5,2.5);  

\node at (-3,1) {respectively};

\end{tikzpicture}
\]
\end{enumerate} 
  \end{lemma}
  \begin{proof}
  We only give a proof between the conditions involving $\log_2 f'(1)$ and the colors of the regions near the root of the bottom tree,
   the other equivalences can be handled in a similar way (otherwise, they also follow from Lemma \ref{operationsonE}-(1)).
  The figure below shows that the color to the left of the rightmost leaf of the top tree
  is $1$ when the length of the path from it to the root has odd length, otherwise it is $2$. 
  
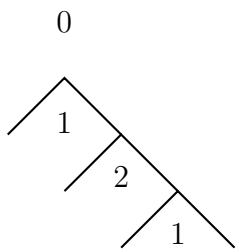
\[
  \begin{tikzpicture}[x=.75cm, y=.75cm,
    every edge/.style={
        draw,
      postaction={decorate,
                    decoration={markings}
                   }
        }
]
\draw[thick] (0,0) -- (1,1)--(2,0);
 \draw[thick] (1,-1)--(2,0)--(3,-1);
  \draw[thick] (2,-2)--(3,-1)--(4,-2);


\node at (1,2) {$0$};
\node at (1,0.2) {$1$};
\node at (2,-.75) {$2$};
\node at (3,-1.75) {$1$};


\end{tikzpicture}
  \]
  Once we reach the rightmost leaf, we continue moving along the shortest path towards the root of the bottom tree. 
  There are two cases to consider: the color to the left of the rightmost leaf is $1$ or $2$.
  The former case occurs when the number of edges from the root to the leaf is odd, 
  the latter when this number is even.
 The bottom tree will look like these, respectively
 
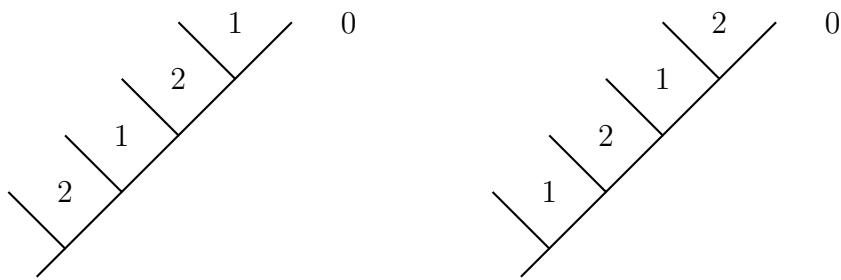
\[
   \begin{tikzpicture}[x=.75cm, y=.75cm,
    every edge/.style={
        draw,
      postaction={decorate,
                    decoration={markings}
                   }
        }
]

\draw[thick] (0,1) -- (1,0)--(2,1);
\draw[thick] (-1,0) -- (0,-1)--(1,0); 
\draw[thick] (-2,-1) -- (-1,-2)--(0,-1); 
\draw[thick] (-3,-2) -- (-2,-3)--(-1,-2); 
\draw[thick] (-2.5,-3.5) -- (-2,-3);


\node at (1,1) {$1$};
\node at (3,1) {$0$};

\node at (0,0) {$2$};
\node at (-1,-1) {$1$};
\node at (-2,-2) {$2$};
 

\end{tikzpicture}
\qquad 
\qquad
   \begin{tikzpicture}[x=.75cm, y=.75cm,
    every edge/.style={
        draw,
      postaction={decorate,
                    decoration={markings}
                   }
        }
]

\draw[thick] (0,1) -- (1,0)--(2,1);
\draw[thick] (-1,0) -- (0,-1)--(1,0); 
\draw[thick] (-2,-1) -- (-1,-2)--(0,-1); 
\draw[thick] (-3,-2) -- (-2,-3)--(-1,-2); 
\draw[thick] (-2.5,-3.5) -- (-2,-3);


\node at (1,1) {$2$};
\node at (3,1) {$0$};

\node at (0,0) {$1$};
\node at (-1,-1) {$2$};
\node at (-2,-2) {$1$};
 

\end{tikzpicture}
 \]
 If   $\log_2 f'(1)\in 2\IZ$, it means that the lengths of the paths from the two roots to the rightmost leaf have the same parity and thus the colors of the regions near the root of the bottom tree are the same as those near the root of the top tree.
 Similarly, when  $\log_2 f'(1)\in 2\IZ+1$,  the lengths of the paths from the two roots to the rightmost leaf have the different parity and thus the color of the region above the root of the bottom tree is $2$. 
   \end{proof}
  Here is an  easy application of the previous lemma.
  \begin{proposition}\label{propintersectionK12}
Let $f\in F$.  If $\log_2f'(0)\not\equiv_2\log_2f'(1)$, then $f\not\in\CE$.
In particular,  it holds   $\CE\cap (K_{(2,1)}\setminus K_{(2,2)})=\emptyset$.  
  \end{proposition}
  
Among other things the next       lemma shows a relation between the elements of the $3$-colorable subgroup $\CF$ 
and those of the planar $3$-colorable subgroup $\CE$.
It will come in handy for determining the generators of $\CE$.
\begin{lemma}\label{operationsonE}
Let $\sigma: F\to F$ be the flip automorphism, that is the order $2$ automorphism  obtained by reflecting tree diagrams about the vertical line,  and let 
$\varphi_R, \varphi_L : F\to  F$ be the right and left shift endomorphisms, which are   defined graphically  as
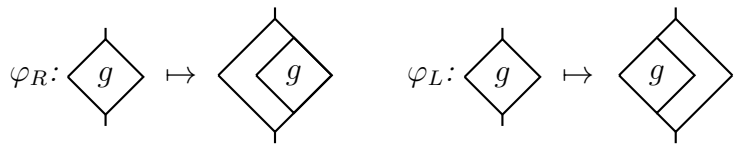
\[
\begin{tikzpicture}[x=1cm, y=1cm,
    every edge/.style={
        draw,
      postaction={decorate,
                    decoration={markings}
                   }
        }
]

\node at (-.45,0) {$\scalebox{1}{$\varphi_R$:}$};

\draw[thick] (0,0)--(.5,.5)--(1,0)--(.5,-.5)--(0,0);
\node at (1.5,0) {$\scalebox{1}{$\mapsto$}$};

\node at (.5,0) {$\scalebox{1}{$g$}$};
\node at (.5,.-.75) {$\scalebox{1}{}$};

 \draw[thick] (.5,.65)--(.5,.5);
 \draw[thick] (.5,-.65)--(.5,-.5);
 
\end{tikzpicture}
\begin{tikzpicture}[x=1cm, y=1cm,
    every edge/.style={
        draw,
      postaction={decorate,
                    decoration={markings}
                   }
        }
]
 
\node at (2.5,0) {$\scalebox{1}{$g$}$};

\draw[thick] (1.5,0)--(2.25,.75)--(3,0)--(2.25,-.75)--(1.5,0);
\draw[thick] (2,0)--(2.5,.5)--(3,0)--(2.5,-.5)--(2,0);
 
 \node at (1.5,.-.75) {$\scalebox{1}{}$};

 \draw[thick] (2.25,.75)--(2.25,.9);
 \draw[thick] (2.25,-.75)--(2.25,-.9);

\end{tikzpicture}
\qquad 
\begin{tikzpicture}[x=1cm, y=1cm,
    every edge/.style={
        draw,
      postaction={decorate,
                    decoration={markings}
                   }
        }
]

\node at (-.45,0) {$\scalebox{1}{$\varphi_L$:}$};

\draw[thick] (0,0)--(.5,.5)--(1,0)--(.5,-.5)--(0,0);
\node at (1.5,0) {$\scalebox{1}{$\mapsto$}$};

\node at (.5,0) {$\scalebox{1}{$g$}$};
\node at (.5,.-.75) {$\scalebox{1}{}$};

 \draw[thick] (.5,.65)--(.5,.5);
 \draw[thick] (.5,-.65)--(.5,-.5);
 
\end{tikzpicture}
\begin{tikzpicture}[x=1cm, y=1cm,
    every edge/.style={
        draw,
      postaction={decorate,
                    decoration={markings}
                   }
        }
]
 
\node at (2.5,0) {$\scalebox{1}{$g$}$};

\draw[thick] (3.5,0)--(2.75,.75)--(2.5,0.5);
\draw[thick] (2.5,-0.5)--(2.75,-.75)--(3.5,0);
\draw[thick] (2,0)--(2.5,.5)--(3,0)--(2.5,-.5)--(2,0);
 
 \node at (1.95,.-.75) {$\scalebox{1}{}$};


 \draw[thick] (2.75,.75)--(2.75,.9);
 \draw[thick] (2.75,-.75)--(2.75,-.9);

\end{tikzpicture}
\]
When $F$ is seen as a group of homeomorphisms of the unit interval, for $g\in F$, these 
maps read as follow 
\begin{align*}
&\varphi_L(g)(t):=\left\{\begin{array}{ll}
t & \text{ if } t\in [0,1/2]\\
g(2t-1) & \text{ if } t\in [1/2 ,1]
\end{array}
\right.\\
&\varphi_R(g)(t):=\left\{\begin{array}{ll}
g(2t) & \text{ if } t\in [0,1/2 ]\\
t & \text{ if } t\in [1/2,1]\\
\end{array}
\right.
\end{align*}
Then 
\begin{enumerate}
\item  $g\in\CE$ if and only if $\sigma(g)\in\CE$;
\item  $g\in\CF$ if and only if $\sigma(g)\in\CF$;
\item $g\in\CF$ if and only if $\varphi_R(g)\in\CE$;
\item $g\in\CF$ if and only if $\varphi_L(g)\in\CE$.
 \end{enumerate}
\end{lemma}
\begin{proof}
For (1), the plane partition associated with $g$ is $3$-colorable if and only if the one associated with $\sigma(g)$ is $3$-colorable since the second plane partition is obtained from the first after a reflection about the vertical line. \\
(2) was proven in \cite[Lemma 1.2]{TV2}.\\
For (3) and (4), 
 the claim is clear by coloring the partitions of the plane corresponding to $g$ and $\varphi_R(g)$,  $\varphi_L(g)$.
 \end{proof}
The  right shift endomorphism  
$\varphi_R$ maps $x_i$ to $x_{i+1}$ for all $i\geq 0$. In particular, 
in terms of homeomorphisms of $[0,1]$, the range of $\varphi_R$ is the subgroup  of $F$ whose elements act trivially on $[0, 1/2]$.
 Similarly, the range of the left shift endomorphism  is the subgroup of $F$ consisting of the elements that act trivially on $[1/2, 1]$.
The right and left shift endomorphisms are clearly injective.

\begin{example}\label{esedue}
By Lemma \ref{operationsonE}, examples of elements of $\CE$ may be obtained as $\varphi_L(\CF)$ and $\varphi_R(\CF)$.
In particular, $\varphi_L(w_i)$, $\varphi_R(w_i)\in \CE$ for $i=0$, $1$, $2$, $3$, where $\{w_i\}_{i=0}^3$ are the generators of $\CF$ (see Figure \ref{genCF} for the pairs of trees representing these elements). 
\end{example}
We are now in a position to determine the generators of $\CE$.
\begin{theorem}\label{genE}
It holds
$\CE= \langle x_0, \varphi_R(\CF)\rangle = \langle x_0, \varphi_L(\CF)\rangle$. 
\end{theorem}
\begin{proof}
First we prove the first equality.
We know by the Examples \ref{esex0} and \ref{esedue} that  $\langle x_0, \varphi_R(\CF)\rangle$ is contained in $\CE$, so we only have to prove the
converse inclusion. Let $g\in \CE$. 
Write $g$ in normal form (see \cite{CFP} for more information about this normal form), that is as 
$g=x_0^{a_0}\cdots x_n^{a_n}x_n^{-b_n}\cdots x_0^{b_0}$ for some $a_0, \ldots , a_n, b_0, \ldots , b_n\in\IN_0$. 
Now $g$ belongs to $\CE$ if and only if $g':=x_0^{-a_0}gx_0^{b_n}$ belongs to $\CE$ since $x_0$ is in $\CE$. 
The element $g'$ is by construction in the image of $\varphi_R$, that is it is of the form $\varphi_R(g'')$ (here $g''$ is nothing but $x_0^{a_1}\cdots x_{n-1}^{a_n}x_{n-1}^{-b_n}\cdots x_0^{b_1}$).
By Lemma \ref{operationsonE}-(3) we have that $g''\in \CF$.

The second equality  follows from the first by applying $\sigma$ to $\langle x_0, \varphi_R(\CF)\rangle$  and using Lemma \ref{operationsonE} (note that $\sigma(x_0)=x_0^{-1}$). 
\end{proof}

By Lemma \ref{lemmarootcol} the logarithm of the derivative at $0$ of the elements of $\CE$ can be either even or odd.
It is natural to consider the elements of the first type because they form a subgroup of $\CE$.
\begin{definition}\label{defCEE}
The \textbf{even part $\CEE$ of the planar $3$-colorable subgroup} is the subgroup defined as
\begin{align}
\CEE&:=\CE\cap K_{(2,2)} \, .
\end{align}
\end{definition}
\begin{example}\label{esex12}
The subgroup $\CEE$ is a proper subgroup of $K_{(2,2)}$. Indeed,
the   element $x_1^2\in K_{(2,2)}$ is not in $\CEE$.
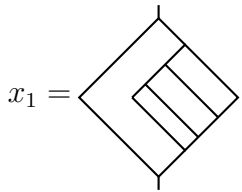
\[
\begin{tikzpicture}[x=.35cm, y=.35cm,
    every edge/.style={
        draw,
      postaction={decorate,
                    decoration={markings}
                   }
        }
]

\node at (-3.5,0) {$\scalebox{1}{$x_1=$}$};
\node at (-1.25,-3.25) {\;};

\draw[thick] (2,2)--(1,3)--(-2,0)--(1,-3)--(2,-2);

\draw[thick] (0,0) -- (2,2)--(4,0)--(2,-2)--(0,0);
 \draw[thick] (.5,.5) --(2.5,-1.5);
 \draw[thick] (1.25,1.25) --(3.25,-.75);

 \draw[thick] (1,3)--(1,3.5);
 \draw[thick] (1,-3)--(1,-3.5);

\end{tikzpicture}
\]
This example shows that in general $\CEE$ and $\CE$ are not preserved by the right shift endomorphism.
\end{example}
Note that thanks to Lemma \ref{operationsonE}, both  subgroups $\CE$ and $\CEE$ are invariant under the action of the flip automorphism $\sigma$.
 \begin{theorem}\label{genE1}
It holds 
 $\CEE= \langle x_0^{2}, \varphi_R(\CF)\rangle= \langle x_0^{2}, \varphi_L(\CF)\rangle$.
  \end{theorem}
\begin{proof} 
We know by the Examples \ref{esex0} and \ref{esedue} that  $\langle x_0^2, \varphi_R(\CF)\rangle$ is contained in $\CEE$, so we only have to prove the
converse inclusion. Let $g\in \CEE$. 
Write $g$ in normal form, that is as 
$g=x_0^{a_0}\cdots x_n^{a_n}x_n^{-b_n}\cdots x_0^{b_0}$ for some $a_0, \ldots$, $a_n$, $b_0, \ldots$, $b_n\in\IN_0$. 
Without loss of generality we may assume that $a_0\geq b_0$ and $g\neq x_0^k$ for $k\in\IZ$.
Note that since $g\in K_{(2,2)}$,   $a_0$ and $b_0$ have the same parity.
Recall that $x_n^{-1}x_0=x_0x_{n+1}^{-1}$ for any $n\geq 1$.
If $a_0$ is even, then $g$ belongs to $\CEE$ if and only if $g':=x_0^{b_0-a_0}(x_0^{-a_0}gx_0^{a_0})=x_{1+(a_0-b_0)}^{a_1}\cdots x_{n+(a_0-b_0)}^{a_n}x_{n+(a_0-b_0)}^{-b_n}\cdots x_{1+(a_0-b_0)}^{b_1}$ does.
If $a_0$ is odd, then $g$ belongs to $\CEE$ if and only if 
\begin{align*}
g'&:=x_0^{b_0-a_0}(x_0^{-a_0-1}gx_0^{a_0+1})\\
&=x_{1+(a_0-b_0+1)}^{a_1}\cdots x_{n+(a_0-b_0+1)}^{a_n}x_{n+(a_0-b_0+1)}^{-b_n}\cdots x_{1+(a_0-b_0+1)}^{b_1}
\end{align*}
does. 
The element $g'$ is by construction in the image of $\varphi_R$, that is it is of the form $\varphi_R(g'')$: in the first case $g''$ is 
$x_{(a_0-b_0)}^{a_1}\cdots x_{n-1+(a_0-b_0)}^{a_n}x_{a_0-b_0+n-1}^{-b_n}\cdots x_{a_0-b_0}^{b_1}$ and in the second 
$x_{(a_0-b_0)+1}^{a_1}\cdots x_{n+(a_0-b_0)}^{a_n}x_{a_0-b_0+n}^{-b_n}\cdots x_{a_0-b_0+1}^{b_1}$ . 
By Lemma \ref{operationsonE}-(3) we have that $g''\in \CF$.

The second equality follows by applying $\sigma$ to $\langle x_0^2, \varphi_R(\CF)\rangle$  and using Lemma \ref{operationsonE} (note that $\sigma(x_0)=x_0^{-1}$). 
\end{proof} 
 We recall the following lemma from \cite[Lemma 4.2]{TV2}.
 \begin{lemma}\label{Lemmamax}
Let $G$ be a subgroup of $F$ such that $\pi(G)=a\IZ\oplus b\IZ$ and $G\neq K_{(a,b)}$. Then, the index of $G$ in $F$ is infinite. 
\end{lemma}
  \begin{proposition}
  It holds
  \begin{enumerate}
\item  the index of $\CEE$ in $\CE$ is $2$,  
\item the index of $\CEE$ in $K_{(2,2)}$ is infinite,
\item the index of $\CEE$ in $F$ is infinite,
\item the index of $\CE$ in $F$ is infinite.
\end{enumerate}
  \end{proposition}
\begin{proof}
(1) It holds $\CE=\CEE\sqcup x_0\CEE$ and so the index is $2$. 
By 
Proposition \ref{propintersectionK12} $\log_2 f'(0)$ and $\log_2f'(1)$ have the same parity. Therefore, either $\log_2 f'(0)$ and $\log_2 f'(1)$ are even, or odd.
In the first case, this means $g\in \CE\cap K_{(2,2)}=\CEE$.
In the second we have that $g$ belong to $(F\setminus K_{(2,2)})\cap \CE$. We may write $g$ as $x_0 (x_0^{-1}g)$,
with $\log_2 g'(0),\log_2 g'(1)\in 2\IZ+1$.
Now observe that $x_0^{-1}g\in K_{(2,2)}$ since
$$
\pi(x_0^{-1}g)=\pi(x_0^{-1})+\pi(g)=(-1,1)+\pi(g)\in 2\IZ\oplus 2\IZ\; .
$$ 

(2) 
By Lemma \ref{operationsonE}
we know that $\varphi_L(w_i)$, $\varphi_R(w_i)\in \CE$ for $i=0$, $1$, $2$, $3$, where $\{w_i\}_{i=1}^4$ are the generators of $\CF$ (see Figure \ref{genCF}).
It is easy to see that $\pi(w_0)=(2,-2)$, $\pi(w_1)=(0,-2)$, $\pi(w_2)=(0,-2)$, $\pi(w_3)=(0,-2)$.
Since $\pi(\varphi_R(w_i))=(0,-2)$ for $i=0, 1, 2, 3$ and $\pi(\varphi_L(x_0^2))=(2,0)$, 
we have $\pi(\CEE)=\pi(K_{(2,2)})=2\IZ\oplus 2\IZ$. 
The subgroup $\CEE$ is a proper subgroup of $K_{(2,2)}$ since $x_1^2$ is not in $\CEE$ (see Example \ref{esex12}).
By Lemma \ref{Lemmamax} the subgroup $\CEE$ of $K_{(2,2)}$ has infinite index.

(4) follows from (1) and (2).
 
(3) follows from (4) and (1).
 \end{proof}

 \begin{remark}
In \cite{TV2} we introduced the subgroup $M_0$ which contains $\CF$  and proved that it is maximal and of infinite index in $K_{(2,2)}$ (in particular, since $F$ and   $K_{(2,2)}$ are isomorphic, we obtained a maximal subgroup of infinite index of $F$ isomorphic to $M_0$).
Since $M_0$ contains $\varphi_R(\CF)$ and $x_0^2$, it is clear that
the subgroup $\CEE$ is a subgroup of $M_0$.
It is easy to see that $x_2^2$ is not in $\CEE$, but it is in $M_0$ \cite{TV2}.
Therefore, $\CEE$ is a proper subgroup of $M_0$. 
\end{remark}
\begin{remark}
The subgroup $\CE$ is not maximal in $F$. Indeed, $\pi(\CE)$ is contained in the  
subgroup $\langle (1,-1),(0,2)\rangle$ of index $2$ in $\IZ\oplus \IZ$.
\end{remark} 

Now we show that
for any positive element $g$ which is not a power of $x_0$, the group $\langle g, \CE\rangle$ is a subgroup of $F$ of index at most $2$.
We recall from Section \ref{sec1} that the positive elements of $F$ are those in the 
 monoid $F_+$ generated by $x_0, x_1, x_2, \ldots$.    
 \begin{proposition}\label{propintersectionpos}
There are no non-trivial positive elements in $\CF$.
 The positive elements of $\CE$ form the monoid generated by $x_0$,
 that is 
$F_+\cap \CE=\{ x_0^n\; | \; n\geq 0 \}$. 
  \end{proposition}
  \begin{proof}
We use the same argument  for both $\CF$ and $\CE$.
  Consider a reduced positive element $(T_+,T_-)$ in $\CF$ or $\CE$.
  We observe that the colors of the inner regions in the bottom tree are always alternating (this can be easily seen by looking at the shape of the bottom tree in a reduced tree diagram
of every positive element). 
  We give a proof by induction on the number $n$ of the leaves  of $T_+$. 
  If $n=1$ or $3$, there is only two elements, namely the identity and $x_0$. The latter belongs to $\CE$, but not to $\CF$.
  Suppose that $n\geq 4$.
The top tree contains one of the following three subtrees (the leaves of theses subtrees are the same as those of the top tree)
 
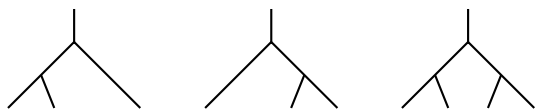
\[
\begin{tikzpicture}[x=1.75cm, y=1.75cm,
    every edge/.style={
        draw,
      postaction={decorate,
                    decoration={markings}
                   }
        }
]

\draw[thick] (0.5,0.75)--(.5,.5);
\draw[thick] (0.35,0)--(.25,.25);
\draw[thick] (0,0)--(.5,.5)--(1,0);

\end{tikzpicture}
\qquad
\begin{tikzpicture}[x=1.75cm, y=1.75cm,
    every edge/.style={
        draw,
      postaction={decorate,
                    decoration={markings}
                   }
        }
]

\draw[thick] (0.5,0.75)--(.5,.5);
\draw[thick] (0.65,0)--(.75,.25);
\draw[thick] (0,0)--(.5,.5)--(1,0);
 
\end{tikzpicture}
\qquad
\begin{tikzpicture}[x=1.75cm, y=1.75cm,
    every edge/.style={
        draw,
      postaction={decorate,
                    decoration={markings}
                   }
        }
]

\draw[thick] (0.5,0.75)--(.5,.5);
\draw[thick] (0.65,0)--(.75,.25);
\draw[thick] (0.35,0)--(.25,.25);
\draw[thick] (0,0)--(.5,.5)--(1,0);

\end{tikzpicture}
 \]
 We will show that in many cases these subtrees cannot occur because the colors determined by the bottom tree are not compatible with  the top tree.
 
 If these subtrees do not contain
 the first or the last leaf of the tree, then the tree diagram is not $3$-colorable  
($i$ and $j$ are two distinct colors) as shown below
 
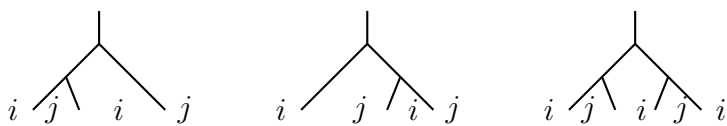
\[
\begin{tikzpicture}[x=1.75cm, y=1.75cm,
    every edge/.style={
        draw,
      postaction={decorate,
                    decoration={markings}
                   }
        }
]

\draw[thick] (0.5,0.75)--(.5,.5);
\draw[thick] (0.35,0)--(.25,.25);
\draw[thick] (0,0)--(.5,.5)--(1,0);

\node at (-0.15,0) {$\scalebox{1}{$i$}$};
\node at (0.15,0) {$\scalebox{1}{$j$}$};
\node at (0.65,0) {$\scalebox{1}{$i$}$};
\node at (1.15,0) {$\scalebox{1}{$j$}$};
\end{tikzpicture}
\qquad
\begin{tikzpicture}[x=1.75cm, y=1.75cm,
    every edge/.style={
        draw,
      postaction={decorate,
                    decoration={markings}
                   }
        }
]

\draw[thick] (0.5,0.75)--(.5,.5);
\draw[thick] (0.65,0)--(.75,.25);
\draw[thick] (0,0)--(.5,.5)--(1,0);

\node at (-0.15,0) {$\scalebox{1}{$i$}$};
\node at (0.45,0) {$\scalebox{1}{$j$}$};
\node at (0.85,0) {$\scalebox{1}{$i$}$};
\node at (1.15,0) {$\scalebox{1}{$j$}$};
\end{tikzpicture}
\qquad
\begin{tikzpicture}[x=1.75cm, y=1.75cm,
    every edge/.style={
        draw,
      postaction={decorate,
                    decoration={markings}
                   }
        }
]

\draw[thick] (0.5,0.75)--(.5,.5);
\draw[thick] (0.65,0)--(.75,.25);
\draw[thick] (0.35,0)--(.25,.25);
\draw[thick] (0,0)--(.5,.5)--(1,0);

\node at (-0.15,0) {$\scalebox{1}{$i$}$};
\node at (0.15,0) {$\scalebox{1}{$j$}$};
\node at (0.55,0) {$\scalebox{1}{$i$}$};
\node at (0.85,0) {$\scalebox{1}{$j$}$};
\node at (1.15,0) {$\scalebox{1}{$i$}$};

\end{tikzpicture}
 \]
 The second and third subtrees cannot have the right-most leaf of the top-tree because otherwise the tree diagram would not be reduced.
The first subtree cannot have the right-most leaf of the top-tree either because it would not be 
 $3$-colorable ($i$, $j$, $k$ are distinct colors)
 
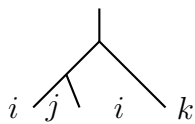
\[
\begin{tikzpicture}[x=1.75cm, y=1.75cm,
    every edge/.style={
        draw,
      postaction={decorate,
                    decoration={markings}
                   }
        }
]

\draw[thick] (0.5,0.75)--(.5,.5);
\draw[thick] (0.35,0)--(.25,.25);
\draw[thick] (0,0)--(.5,.5)--(1,0);

\node at (-0.15,0) {$\scalebox{1}{$i$}$};
\node at (0.15,0) {$\scalebox{1}{$j$}$};
\node at (0.65,0) {$\scalebox{1}{$i$}$};
\node at (1.15,0) {$\scalebox{1}{$k$}$};
\end{tikzpicture}\]
Now it remains to consider the cases when these subtrees have the left-most leaf of the top tree  ($i$, $j$, $k$ are distinct colors).
 
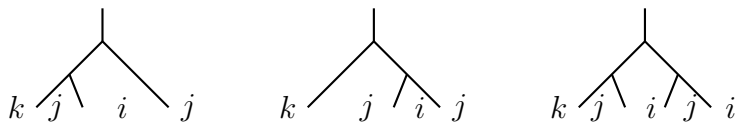
\[
\begin{tikzpicture}[x=1.75cm, y=1.75cm,
    every edge/.style={
        draw,
      postaction={decorate,
                    decoration={markings}
                   }
        }
]

\draw[thick] (0.5,0.75)--(.5,.5);
\draw[thick] (0.35,0)--(.25,.25);
\draw[thick] (0,0)--(.5,.5)--(1,0);

\node at (-0.15,0) {$\scalebox{1}{$k$}$};
\node at (0.15,0) {$\scalebox{1}{$j$}$};
\node at (0.65,0) {$\scalebox{1}{$i$}$};
\node at (1.15,0) {$\scalebox{1}{$j$}$};
\end{tikzpicture}
\qquad
\begin{tikzpicture}[x=1.75cm, y=1.75cm,
    every edge/.style={
        draw,
      postaction={decorate,
                    decoration={markings}
                   }
        }
]

\draw[thick] (0.5,0.75)--(.5,.5);
\draw[thick] (0.65,0)--(.75,.25);
\draw[thick] (0,0)--(.5,.5)--(1,0);

\node at (-0.15,0) {$\scalebox{1}{$k$}$};
\node at (0.45,0) {$\scalebox{1}{$j$}$};
\node at (0.85,0) {$\scalebox{1}{$i$}$};
\node at (1.15,0) {$\scalebox{1}{$j$}$};
\end{tikzpicture}
\qquad
\begin{tikzpicture}[x=1.75cm, y=1.75cm,
    every edge/.style={
        draw,
      postaction={decorate,
                    decoration={markings}
                   }
        }
]

\draw[thick] (0.5,0.75)--(.5,.5);
\draw[thick] (0.65,0)--(.75,.25);
\draw[thick] (0.35,0)--(.25,.25);
\draw[thick] (0,0)--(.5,.5)--(1,0);

\node at (-0.15,0) {$\scalebox{1}{$k$}$};
\node at (0.15,0) {$\scalebox{1}{$j$}$};
\node at (0.55,0) {$\scalebox{1}{$i$}$};
\node at (0.85,0) {$\scalebox{1}{$j$}$};
\node at (1.15,0) {$\scalebox{1}{$i$}$};

\end{tikzpicture}
 \]
 In the second and third case, the subtree cannot be contained in the tree diagram because the corresponding regions of the plane are not $3$-colorable.

For the first subtree we divide our proof into two parts because the present proposition has two statements: one for $\CF$ and one for $\CE$.
For $\CF$, we want to show that actually this subtree cannot be contained in the top tree and thus the only positive element of $\CF$ is the trivial element (because we have shown that there is no reduced tree diagram with more than $1$ leaf).
The element does not belong to $\CF$ because the first four regions have colors  $k$, $j$, $i$, $j$ (with $i, j, k,$ distinct),  
but at the same time the region to the right of the root of the bottom tree must have color $k'$ different from  $k$, $j$, $i$, $j$, which is impossible (see the figure below).
 
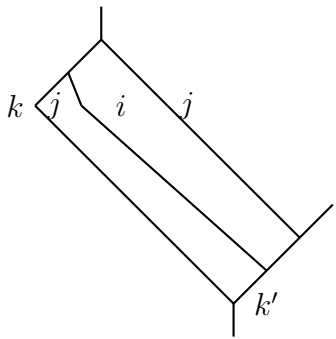
\[
\begin{tikzpicture}[x=1.75cm, y=1.75cm,
    every edge/.style={
        draw,
      postaction={decorate,
                    decoration={markings}
                   }
        }
]

\draw[thick] (0.5,0.75)--(.5,.5);
\draw[thick] (0.35,0)--(.25,.25);
\draw[thick] (0,0)--(.5,.5)--(1,0);

\draw[thick] (0,0)--(1.5,-1.5)--(2.25,-.75);
\draw[thick] (0.35,0)--(1.75,-1.25);
\draw[thick] (1,0)--(2,-1);
\draw[thick] (1.5,-1.5)--(1.5,-1.75);

\node at (1.75,-1.5) {$\scalebox{1}{$k'$}$};
\node at (-0.15,0) {$\scalebox{1}{$k$}$};
\node at (0.15,0) {$\scalebox{1}{$j$}$};
\node at (0.65,0) {$\scalebox{1}{$i$}$};
\node at (1.15,0) {$\scalebox{1}{$j$}$};
\end{tikzpicture}
\]
It remains to consider the case when $(T_+,T_-)$ is in $\CE$.
In this case,  we  multiply $(T_+,T_-)$ by $x_0^{-1}\in \CE$ and obtain a positive tree diagram $(T_+',T_-')$, where $T'$ has $n-2$ leaves and for which we may apply the inductive argument.
  \end{proof}

We give a small improvement to  \cite[Proposition 5.4]{TV2}. 
\begin{lemma}\label{propposCE}
Let $g\in F_+$, then $\langle g,\CF\rangle$ contains $M_0:=\langle x_0^2, \CF\rangle$, $M_1:=\langle x_1^2, \CF\rangle$, or $K_{(2,2)}$.
\end{lemma}
\begin{proof}
If $g\in K_{(2,2)}$ this is exactly the content of \cite[Proposition 5.4]{TV2}. 
Otherwise, suppose that $g\not\in K_{(2,2)}$ and
write $g$ in its normal form as $x_0^{a_0}\cdots x_n^{a_n}$. This means that either $\sum_{i=0}^n a_i\equiv_2 1$ or $a_0\equiv_2 1$.
The element $g^2$ is a positive element of $K_{(2,2)}$.
By Proposition \ref{propintersectionpos} we know that $g^2$ is not $\CF$.
Since $\langle g^2, \CF\rangle \subset \langle g, \CF\rangle$,
by applying  \cite[Proposition 5.4]{TV2}  we have that 
$\langle g^2, \CF\rangle $ contains one among $M_0$, $M_1$ and $K_{(2,2)}$ and we are done. 
\end{proof}
\begin{proposition}
For any $g\in F_+\setminus \CE$, the subgroup $\langle g, \CE\rangle$ contains the subgroup 
$$
H:=\{f\in F\; | \; \log_2f'(0)\equiv_2 \log_2f'(1)		\}\; 
$$
where $\equiv_2$ denotes equivalence modulo $2$.
\end{proposition}
\begin{proof}
By Proposition \ref{propintersectionpos} we know that $g$ is not a power of $x_0$.
Without loss of generality we may assume that $g\in\varphi_R(F_+)$, i.e. $g=\varphi_R(h)$ for some $h\in F_+$. Indeed, write $g$ in its normal form, 
i.e. $g=x_0^{a_0}\cdots x_n^{a_n}$, then $g':=x_0^{-a_0}g\in\varphi_R(F_+)$
and $\langle g, \CE\rangle =\langle g', \CE\rangle$.
We observe that, by Lemma \ref{operationsonE} we have
 $\varphi_R(\langle h, \CF\rangle)=\langle \varphi_R(h), \varphi_R(\CF)\rangle\subset \langle g, \CE\rangle$.

By Lemma \ref{propposCE} there are three cases to consider:  
$\langle h, \CF\rangle$ contains $M_0$, 
$\langle h, \CF\rangle$ contains  $M_1$,
and $\langle h, \CF\rangle$ contains $K_{(2,2)}$.
We will use the following facts from \cite{TV2}: $M_0$ and $M_1$ are maximal in $K_{(2,2)}$, $x_2^2\in M_0$, $x_3^2\in M_1$.

We first show that, in all three cases, the subgroup $\langle g,\CE\rangle$ contains $\varphi_R(K_{(2,2)})$.

In the first case, we know that, since $x_2^2\in M_0$ (\cite[Proposition 4.11]{TV2}),  we have
$\varphi_R(x_2^2)\in \langle g,\CE\rangle$. 
As $x_0\in \CE$, 
we also have
$$
x_0^{-1}x_3^2 x_0=x_0^{-1}\varphi_R(x_2^2) x_0=x_4^2=\varphi_R(x_3^2)\in\langle g,\CE\rangle
$$ 
Maximality of $M_0$ in $K_{(2,2)}$ 
 implies that 
$$
\langle \varphi_R(M_0), \varphi_R(x_3^2)\rangle = \varphi_R(\langle M_0, x_3^2\rangle)=\varphi_R(\langle K_{(2,2)}\rangle)\; .
$$ 

Similarly in the second case, 
since $x_3^2\in M_1$, 
we have
$\varphi_R(x_3^2)\in \langle g,\CE\rangle$. 
As $x_0\in \CE$, 
we also have
$$
x_0^{-1}x_4^2 x_0=x_0^{-1}\varphi_R(x_3^2) x_0=x_5^2=\varphi_R(x_4^2)\in\langle g,\CE\rangle
$$
and by the maximality of $M_1$ in $K_{(2,2)}$ 
we have
$$
\langle \varphi_R(M_1), \varphi_R(x_4^2)\rangle = \varphi_R(\langle M_1, x_4^2\rangle)=\varphi_R(\langle K_{(2,2)}\rangle)\; .
$$ 

In the last case there is nothing to do.

From this point we can assume that $\varphi_R(K_{(2,2)})\subset \langle g,\CE\rangle$.
Recall from \cite[Lemma 4.6]{GS} that
$$\langle x_0x_1, x_1x_2, x_2x_3\rangle =\vec{F}$$ 
and that $\vec{F}$ is maximal in $K_{(1,2)}$ \cite[Theorem 3.12]{GS2}.

We know that $x_2x_3=\varphi_R(x_1x_2)$, $x_3x_4=\varphi_R(x_2x_3), x_2^2\in \varphi_R(K_{(2,2)})$ and that
 $x_1x_2=x_0(x_2x_3)x_0^{-1}\in \langle g,\CE\rangle$.
Then we have that 
$$
\varphi_R(K_{(1,2)})=\langle \varphi_R(\vec{F}),\varphi_R(x_1^2)\rangle\subset \langle g,\CE\rangle
$$
Now $\varphi_R(K_{(1,2)})$ consists of all the elements $y\in F$ whose normal form has even length and $a_0=b_0=0$.
The rectangular subgroup $K_{(2,2)}$ consists of the elements of $F$ whose normal form $x_0^{a_0}x_1^{a_1}\cdots x_n^{a_n}x_n^{-b_0}\cdots x_1^{-b_1}x_0^{-b_0}$
is such that: $a_0$ and $b_0$ have the same parity; $\sum_{i\geq 1} (a_i+b_i)$ is even. 
As $x_0\in \CE$, we have that any element of $K_{(2,2)}$ is contained in $\langle g,\CE\rangle$.
Finally, we get $\langle x_0, K_{(2,2)}\rangle = H$ because the index of $K_{(2,2)}$ in $H$ is $2$.
\end{proof}

 We would like to end this section discussing another interesting subgroup of $F$ defined similarly to $\CE$ that happens to coincide with the subgroup generated by $\CE$ and $\CF$.
Recall that   the definition of $\CF$ 
was based on the $3$-colorability of a strip partition (see the diagram in the middle in the figure below), 
 while for $\CE$ 
 we considered the $3$-colorability of a plane partition (as in the leftmost diagram in the figure below). 
It is natural to also  consider the following intermediate case: 
take the strip 
 passing through the roots of the two trees  and consider the group generated by the elements of $F$ for which this smaller strip is $3$-colorable. We denote this group by $\tilde{\CE}$.
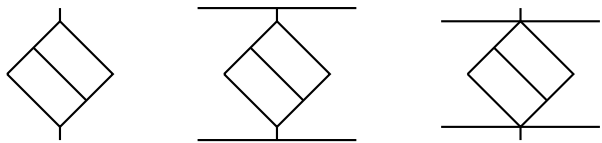
\[
\begin{tikzpicture}[x=.35cm, y=.35cm,
    every edge/.style={
        draw,
      postaction={decorate,
                    decoration={markings}
                   }
        }
]

 \node at (-1.25,-3) {\;};

\draw[thick] (0,0) -- (2,2)--(4,0)--(2,-2)--(0,0);
 \draw[thick] (1,1) -- (2,0)--(3,-1);

 \draw[thick] (2,2)--(2,2.5);

 \draw[thick] (2,-2)--(2,-2.5);

\end{tikzpicture}
\qquad
\begin{tikzpicture}[x=.35cm, y=.35cm,
    every edge/.style={
        draw,
      postaction={decorate,
                    decoration={markings}
                   }
        }
]

 \draw[thick] (-1,2.5)--(5,2.5);
 \draw[thick] (-1,-2.5)--(5,-2.5);

 \node at (-1.25,-3) {\;};

\draw[thick] (0,0) -- (2,2)--(4,0)--(2,-2)--(0,0);
 \draw[thick] (1,1) -- (2,0)--(3,-1);

 \draw[thick] (2,2)--(2,2.5);

 \draw[thick] (2,-2)--(2,-2.5);

\end{tikzpicture}
\qquad
\begin{tikzpicture}[x=.35cm, y=.35cm,
    every edge/.style={
        draw,
      postaction={decorate,
                    decoration={markings}
                   }
        }
]

 \draw[thick] (-1,2)--(5,2);
 \draw[thick] (-1,-2)--(5,-2);

 \node at (-1.25,-3) {\;};

\draw[thick] (0,0) -- (2,2)--(4,0)--(2,-2)--(0,0);
 \draw[thick] (1,1) -- (2,0)--(3,-1);

 \draw[thick] (2,2)--(2,2.5);

 \draw[thick] (2,-2)--(2,-2.5);

\end{tikzpicture}
\]

We are now in a position to give a complete description of $\tilde\CE$.
\begin{proposition}
The subgroup $\tilde\CE$ coincides with the index $2$ subgroup of $F$
$$
H=\{f\in F\; | \; \log_2f'(0)\equiv_2 \log_2f'(1)		\}
$$
where $\equiv_2$ denotes equivalence modulo $2$.
\end{proposition}
\begin{proof}
We first prove that it holds $\tilde\CE= \langle \CE, \CF\rangle$.
By definition $\tilde \CE$
is generated by the elements of $F$ for which the strip  drawn as
in the rightmost diagram
 in the figure above is $3$-colorable.
There are two cases: the colors of the leftmost and rightmost regions are the same or not.
In the first case the element is in $\CE$, while in the second it is in $\CF$.
 The converse inclusion is obvious.

Then we prove that the subgroup $\tilde\CE$ contains the rectangular subgroup $K_{(2,2)}$.
By the previous 
argument we know that $\tilde\CE$ contains $\CF$ and $x_0^2$. 
In particular, $\tilde \CE$ contains the subgroup $M_0=\langle \CF, x_0^2\rangle$.
By \cite[Prop. 4.11]{TV2} $x_2^2$ is in $M_0\subset \tilde\CE$.
As $x_0\in\tilde\CE$, it follows that $x_1^2=x_0x_2^2x_0^{-1}\in \tilde\CE\cap K_{(2,2)}$.
By the maximality of $M_0$ in $K_{(2,2)}$,
\cite[Theorem 4.6]{TV}, we have that $\tilde\CE\cap K_{(2,2)}$ contains $K_{(2,2)}$.

 Recall that the map $\pi$ provides a bijection 
between finite index subgroups of $F$ and those of $\IZ\oplus \IZ$, \cite{BW}.
The subgroup $H$ is a finite index subgroup of $F$ because it can be described as $\pi^{-1}(\{(a,b)\in \IZ\oplus\IZ\; | \; a\equiv_2 b	\})$. 

By 
previous discussion, it is clear that $\tilde \CE$ is contained in $H$. 

For the converse inclusion, take $f\in H$. If $\log_2f'(0)\equiv_2 \log_2f'(1)\equiv_2 0$, then $f$ is in $K_{(2,2)}$ and 
 $f$ is also in $\tilde\CE$.
Otherwise, if $\log_2f'(0)\equiv_2 \log_2f'(1)\equiv_2 1$, consider $fx_0$. 
Now $\log_2(fx_0)'(0)\equiv_2 \log_2 (fx_0)'(1)\equiv_2 0$, so $fx_0\in \tilde\CE$ by the previous argument. 
Since $x_0\in\tilde\CE$, it holds $f\in\tilde\CE$ as well.
\end{proof}

\section{Quasi-regular representations and closed subgroups}\label{sec3}
 \label{sec3}
In the first half of this section we investigate how $\CE$ and its even part $\CEE$ act on the dyadic rationals.
In the second half we use this information to study the (ir)reducibility of quasi-regular representations associated with $\CEE$.
 
Similar analysis for the $3$-colorable subgroup $\CF$ can be found
 in the predecessor of this paper \cite{TV2}, 
  and for  the oriented subgroup $\vec{F}$ in an earlier paper 
by Golan and Sapir  \cite{GS}.
The oriented subgroup was realized as the stabiliser of the subset of dyadic rationals consisting of elements with an even number of digits  equal to $1$ in the binary expansion. 
The  $3$-colorable subgroup $\CF$ was described as  the intersection of the stabilizers of  three distinct subsets of dyadic rationals that were defined in terms of a suitable \emph{weight} function $\omega$.

For $\CEE$ we   introduce a new weight function that we use to define three subsets $Z_1$, $Z_2$, and $Z_3$ of the dyadic rationals and show that  $\CEE$ is  the intersection of the stabilisers of these subsets.
 An analogous study  
for the ternary oriented subgroup $\vec{F}_3$ of $F_3$ was performed in \cite{TV}. 

We  
first describe the action of $\CF$.
Let us define the subsets\footnote{
Note that these subsets were originally defined in \cite{TV2}, but there was an unnecessary condition added mistakenly on the parity of length of the binary expansion of the dyadic numbers.}
\begin{align*}
S_i&=\{t \in (0,1)\cap \IZ[1/2]\; | \; \omega(t)\equiv_3 i \} 
\end{align*}
where $i\in \IZ_3$,
  $\equiv_3$ is the equivalence modulo $3$, 
  and
  the weight function is defined by
$\omega(.a_1\ldots a_n):= \sum_{j=1}^n (-1)^j a_j$ (mod $3$). These subsets form a partition of the dyadic rationals.
\begin{lemma}\label{actionFonSi}
For all $i\in\IZ_3$, $\CF$ preserves
$S_i$   and acts transitively on each of them.
\end{lemma}
\begin{proof}
From   \cite[Proposition 2.5]{TV2}, it follows that by acting on $t$ by an element $g$ in $\CF$ we obtain $g(t)$ with $\omega(g(t))\equiv_3\omega(t)$.
This means that $S_i$ are preserved by $\CF$.
We now show that $\CF$ acts transitively on them.

Let $.\alpha, .\beta\in S_i$.  We want to find an element $g\in\CF$ mapping $.\alpha$ to $.\beta$ and we will exhibit a tree diagram representing it.
Let $h:=\max\{|.\alpha|, |.\beta|\}$
and assume that it is equal to $|. \alpha|$. Take two copies of the tree $\Phi(T_{h})$, where  $T_{h}$ is the full $4$-ary tree of height $h$, and call them
 $T'_+$ and $T'_-$. We are going to modify $T'_+$ and $T'_-$ in such a way that they map $.\alpha$ to $.\beta$ and, at the same time, it is clear that they are in the image of $\Phi$ (therefore, by \cite{Ren} the pair of trees belongs to $\CF$).
 
 Recall that $\Phi$ is the map  defined in Figure \ref{fig-ren-map-2}. These trees contain leaves whose corresponding words are $\alpha 0^l$ and $\beta 0^m$ for some  $l, m\in\IN_0$.
Number the leaves of $T'_\pm$ from left to right.  Recall that by hypothesis   $\omega(\alpha)=\omega(\beta)$.  

By induction it can be shown that, in the full  binary tree of height $2h$ ($h\geq 2$), the number of leaves to left of the leaves with the same weight is the same modulo $3$. 
Indeed, if we have the following colorings in the regions to the left and right of a leaf and we add a basic tree, the colors remain the same when the index number of the leaves are the same modulo $3$
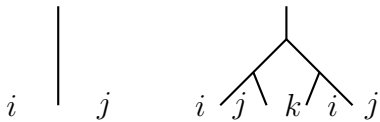
\[
\begin{tikzpicture}[x=1.75cm, y=1.75cm,
    every edge/.style={
        draw,
      postaction={decorate,
                    decoration={markings}
                   }
        }
]

\draw[thick] (0.5,0.75)--(.5,0);

 \node at (0.15,0) {$\scalebox{1}{$i$}$};
\node at (0.85,0) {$\scalebox{1}{$j$}$};

\end{tikzpicture}
\qquad
\begin{tikzpicture}[x=1.75cm, y=1.75cm,
    every edge/.style={
        draw,
      postaction={decorate,
                    decoration={markings}
                   }
        }
]

\draw[thick] (0.5,0.75)--(.5,.5);
\draw[thick] (0.65,0)--(.75,.25);
\draw[thick] (0.35,0)--(.25,.25);
\draw[thick] (0,0)--(.5,.5)--(1,0);

\node at (-0.15,0) {$\scalebox{1}{$i$}$};
\node at (0.15,0) {$\scalebox{1}{$j$}$};
\node at (0.55,0) {$\scalebox{1}{$k$}$};
\node at (0.85,0) {$\scalebox{1}{$i$}$};
\node at (1.15,0) {$\scalebox{1}{$j$}$};

\end{tikzpicture}
\]
In particular, the leaves  to the left of those corresponding to $\alpha 0^l$ and $\beta 0^m$ have the same weight modulo $3$. 
This means that after attaching copies of the basic tree below the leftmost leaf of $T'_+$ or $T'_-$ we have that the leaves corresponding to  $\alpha 0^l$ and $\beta 0^m$ 
can have the same number. We can use the same trick with right-most leaves to make sure that the two trees have the same number of leaves. Let $T_+$ and $T_-$ be the trees obtained 
from $T'_+$ and $T'_-$ with this procedure, respectively. The element $(T_+,T_-)$
is the element $g$ we were looking for.
\end{proof} 
 Thanks to the knowledge of the orbits of $\CF$ and of a generating set of $\CE$ we will derive the next result.
\begin{proposition}\label{actiontransitive}
For any $i\in\IZ_3$, set $S_{i,1}:=S_i\cap [0,1/2]$,  $S_{i,2}:=S_i\cap [1/2,1]$. 
The subgroup $\CE$ has two orbits on the dyadic rationals, namely $A:=S_{0,1}\cup S_{1,2}$ and $B:=S_{1,1}\cup S_{2,1}\cup S_{2,2}\cup S_{0,2}$.
Moreover, $\CE$ is a proper subgroup of Stab$(A)$ $\cap$ Stab$(B)$.
\end{proposition}
\begin{proof}
By Lemma \ref{actionFonSi} it follows that $\varphi_R(\CF)\subset \CE$  acts transitively on $S_{i,2}$  for all $i\in\IZ_3$.
Similarly, $\varphi_L(\CF)\subset \CE$  acts transitively on each $S_{i,1}$.
By Theorem \ref{genE} and the previous observation,  the only way to move from one of these subsets to another is to apply $x_0$.
Here, we describe how the weight changes when we apply $x_0$ on the (smallest) dyadic partition of its domain, namely on $[0,1/4]$, $[1/2,1/2]$, $[1/2,1]$
\begin{enumerate}
\item for $t=.00\alpha$, we have $\omega(x_0(t))=\omega(.0\alpha)=-\omega(t)$;
\item for $t=.01\alpha$, we have $\omega(x_0(t))=\omega(.10\alpha)=\omega(t)+1$: indeed, $\omega(t)=\omega(\alpha)+1$ and $\omega(x_0(t))=\omega(.10\alpha)=\omega(\alpha)-1$;
\item for $t=.1\alpha$, we have $\omega(x_0(t))=\omega(.11\alpha)=-\omega(t)-1$: indeed, $\omega(t)=-1-\omega(\alpha)$ and $\omega(x_0(t))=\omega(.11\alpha)=\omega(.\alpha)$
\end{enumerate}
Therefore, by (3) we have that: $x_0(S_{0,2})\subset S_{2,2}$, $x_0(S_{1,2})\subset S_{1,2}$, $x_0(S_{2,2})\subset S_{0,2}$.
By (1) and (2) we have that $x_0(S_{0,1}\cap [0,1/4])\subset S_{0,1}$, $x_0(S_{0,1}\cap [1/4,1/2])\subset S_{1,2}$,
 $x_0(S_{1,1}\cap [0,1/4])\subset S_{2,1}$, $x_0(S_{1,1}\cap [1/4,1/2])\subset S_{2,2}$,
  $x_0(S_{2,1}\cap [0,1/4])\subset S_{1,1}$, $x_0(S_{2,1}\cap [1/4,1/2])\subset S_{0,2}$ 
  and for the first statement we are done.
  
  We now prove that $\CE$ is a proper subgroup of Stab$(A)$ $\cap$ Stab$(B)$. 
  To this end, it suffices to show that 
  $k:=x_1x_2^{-1}x_1^{-1}x_0^{-1}$ is in Stab$(A)$ $\cap$ Stab$(B)$, but not in $\CE$.
  This element acts as follows
  $$
k(t) = \left\{ \begin{array}{ll} 
.00\alpha & \text{if }  t= .0\alpha\\
.010\alpha & \text{if }  t= .100\alpha\\
.0110\alpha & \text{if }  t= .101\alpha\\
.0111\alpha & \text{if }  t= .110\alpha\\
.1\alpha & \text{if }  t= .111\alpha\\
\end{array}
\right. 
  $$
  There are $5$ intervals in the minimal standard dyadic partition of the domain, namely $[0,1/2]$, 
  $[1/2,5/8]$,
$[5/8,3/4]$,
$[3/4,7/8]$,
$[7/8,1]$, 
 where the weight changes as follows
  \begin{enumerate}
\item[(4)] for $t=.0\alpha$, we have $\omega(k(t))=\omega(.\alpha)=-\omega(t)$;
\item[(5)] for $t=.100\alpha$, we have $\omega(k(t))=\omega(.010\alpha)=\omega(t)+2$: indeed, $\omega(t)=-\omega(\alpha)-1$ and $\omega(k(t))=-\omega(\alpha)+1$;
\item[(6)] for $t=.101\alpha$, we have $\omega(k(t))=\omega(.0110\alpha)=-\omega(t)+1$: indeed, $\omega(t)=-\omega(\alpha)-2$ and $\omega(k(t))=\omega(\alpha)$;
\item[(7)] for $t=.110\alpha$, we have $\omega(k(t))=\omega(.0111\alpha)=-\omega(t)+1$: indeed, $\omega(t)=-\omega(\alpha)$ and $\omega(k(t))=\omega(\alpha)+1$;
\item[(8)] for $t=.111\alpha$, we have $\omega(k(t))=\omega(.111\alpha)=\omega(t)$.
\end{enumerate}
  By (4) we have $k(S_{0,1}\cap [0,1/2])\subset S_{0,1}$, $k(S_{1,1}\cap [0,1/2])\subset S_{2,1}$, $k(S_{2,1}\cap [0,1/2])\subset S_{1,1}$, 
by (5) we have $k(S_{0,2}\cap [1/2,5/8])\subset S_{2,1}$, $k(S_{1,2}\cap [1/2,5/8])\subset S_{0,1}$, $k(S_{2,2}\cap [1/2,5/8])\subset S_{1,1}$,
by (6) we have $k(S_{0,2}\cap [5/8, 3/4])\subset S_{1,1}$,  $k(S_{1,2}\cap [5/8, 3/4])\subset S_{0,1}$, $k(S_{2,2}\cap [5/8, 3/4])\subset S_{2,1}$,
by (7) we have $k(S_{0,2}\cap [3/4,7/8])\subset S_{1,1}$,  $k(S_{1,2}\cap [3/4,7/8])\subset S_{0,1}$, $k(S_{2,2}\cap [3/4,7/8])\subset S_{2,2}$,
by (8) we have $k(S_{0,2}\cap [7/8,1])\subset S_{0,2}$,  $k(S_{1,2}\cap [7/8,1])\subset S_{1,2}$, $k(S_{2,2}\cap [7/8,1])\subset S_{2,2}$.
  \end{proof}
%
   
   We now work towards a description of $\CEE$ (the even part of $\CE$, Definition \ref{defCEE}) in terms of stabiliser subgroups. 
Consider a   rooted binary planar tree $T$. We draw it in the upper-half plane with its leaves on the $x$-axis   and the root on the line $y=1$.
Given a vertex of  a tree, there exists a unique minimal path from the root of the tree to the vertex. This path is made by a collection of left and right edges, and may be represented by a word  in the letters $\{0,1\}$ ($0$ stands for a left edge, $1$ for a right edge). 
An easy inductive argument on the number of vertices shows that
the upper-half plane is always $3$-colorable. 
We adopt the same convention as 
in Convention \ref{conventionCE}, that is, the region above of the root is colored with $0$, the region 
below the root is coloured
 with $1$. After this choice, there is a unique coloring of the strip.

Given a vertex $v$ of $T$, we denote by $\tilde\omega(v)$ the color of the region to the left of $v$. 
We call $\tilde\omega(\cdot)$ the weight associated with $T$.
The weight $\tilde\omega$ can be actually defined on the infinite binary rooted planar tree 
$\CT_2$.
Any rooted planar binary tree can be seen as rooted sub-tree of $\CT_2$ and  the restriction of the weight of $\CT_2$ to the vertices of $T$ is then the weight of $\CT_2$. 
 Vertices of $\CT_2$ can be viewed as finite binary words $\{0, 1\}^*$.
Therefore, 
we have a function $\tilde\omega: \{0, 1\}^*\to \IZ_3$.
When we consider a tree diagram $(T_+,T_-)$, we denote by $\tilde\omega_+(\cdot)$ and $\tilde\omega_-(\cdot)$ the weights associated with the top tree and the bottom tree (after a reflection about the $x$-axis), respectively.

The next  lemma 
presents some properties that allow one to calculate the weight function.
 \begin{lemma}\label{lemma1} \label{lemmaweightsprop}
For all $\alpha$, $\beta\in \{0, 1\}^*$, and $n\in\IN$,
it holds 
\begin{align}
&\tilde\omega (0\alpha)=\omega(\alpha) \label{formula1A}\\
&\tilde\omega(1\alpha)=\omega(101\alpha)=1-\omega(\alpha) \label{formula1B}\\
& \tilde\omega(\alpha 00 \beta)=\tilde\omega(\alpha \beta) \label{formula2}\\
& \tilde\omega(\alpha 11 \beta)=\tilde\omega(\alpha \beta)\label{formula3}\\
& \tilde\omega(\alpha 0)=\tilde\omega(\alpha) \label{formula4}\\
\label{formula1}
& \tilde\omega((10)^n)=\left\{ \begin{array}{cc} 
1 & \text{if }  n= 1+3k\\
0 & \text{if }  n= 2+3k\\
2 & \text{if } n= 3+3k
\end{array}
\right. 
\qquad 
\tilde\omega((01)^n)=\left\{ \begin{array}{cc} 
 2 & \text{if }  n= 1+3k\\
1 & \text{if }  n= 2+3k\\
0 & \text{if } n= 3+3k 
\end{array}
\right. 
\end{align}
where $k\geq0$ and, for a finite word $w$ in $0$ and $1$, $w^n$ denotes the word obtained by concatenating $n$ copies of $w$.  
\end{lemma}
\begin{proof}
It suffices to prove Formulas \ref{formula1A} and \ref{formula1B}, then the others follow from  \cite[Lemma 2.2, Lemma 2.3]{TV2}.
The first one follows from the coloring the  trees of the standard dyadic intervals according to the conventions for $\CE$ (on the left) and for $\CF$ (on the right).
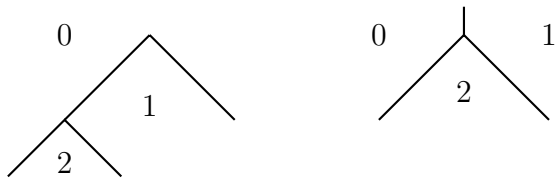
\[
\begin{tikzpicture}[x=.75cm, y=.75cm,
    every edge/.style={
        draw,
      postaction={decorate,
                    decoration={markings}
                   }
        }
]
\draw[thick] (0,0) -- (1.5,1.5)--(3,0);
 \draw[white] (1.5,1.5) -- (1.5,2);
 \draw[thick] (-1,-1)--(0,0)--(1,-1);

 

\node at (0,1.5) {$0$};
\node at (1.5,0.25) {$1$};
\node at (0,-.75) {$2$};

\node at (1,-5) {$\;$};

\end{tikzpicture}
\qquad\qquad
\begin{tikzpicture}[x=.75cm, y=.75cm,
    every edge/.style={
        draw,
      postaction={decorate,
                    decoration={markings}
                   }
        }
]
\draw[thick] (0,0) -- (1.5,1.5)--(3,0);
 \draw[thick] (1.5,1.5) -- (1.5,2);
%
 

\node at (0,1.5) {$0$};
\node at (3,1.5) {$1$};
\node at (1.5,.5) {$2$};

\node at (1,-5) {$\;$};

\end{tikzpicture}
\]
The second  follows by drawing the trees from $\CE$ (on the left) and from $\CF$ (on the right).
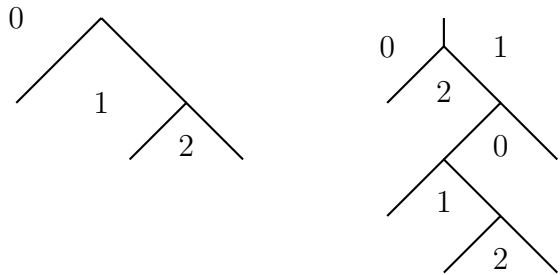
\[
\begin{tikzpicture}[x=.75cm, y=.75cm,
    every edge/.style={
        draw,
      postaction={decorate,
                    decoration={markings}
                   }
        }
]
\draw[thick] (0,0) -- (1.5,1.5)--(3,0);
\draw[white] (1.5,1.5) -- (1.5,2);

 \draw[thick] (2,-1)--(3,0)--(4,-1);
 

\node at (0,1.5) {$0$};
\node at (1.5,0) {$1$};
\node at (3,-.75) {$2$};

\node at (1,-5) {$\;$};

\end{tikzpicture}
\qquad
\qquad
\begin{tikzpicture}[x=.75cm, y=.75cm,
    every edge/.style={
        draw,
      postaction={decorate,
                    decoration={markings}
                   }
        }
]
\draw[thick] (0,0) -- (1,1)--(2,0);
 \draw[thick] (1,1) -- (1,1.5);
 \draw[thick] (1,-1)--(2,0)--(3,-1);
 \draw[thick] (0,-2)--(1,-1)--(2,-2);
 \draw[thick] (1,-3)--(2,-2)--(3,-3);
 

\node at (0,1) {$0$};
\node at (2,1) {$1$};
\node at (1,0.2) {$2$};
\node at (2,-.75) {$0$};
 
\node at (1,-1.75) {$1$};
\node at (2,-2.75) {$2$};

\node at (1,-5) {$\;$};

\end{tikzpicture}
\]
\end{proof}
 \begin{proposition}\label{prop-descr-3col}
It holds
$$
\CEE=\CE\cap K_{(2,2)}=\{(T_+,T_-)\in K_{(2,2)}\; |\; \tilde\omega_+(i)=\tilde\omega_-(i) \; \forall i\geq 0\}
$$
where $\tilde\omega_+(i)$ and $\tilde\omega_-(i)$  
stand for the colors associated with the $i$-th leaf of the trees $T_+$ and $T_-$, respectively.
\end{proposition}
 \begin{proof}
 The claim follows from Lemma \ref{lemmarootcol} and from the definition of $\tilde\omega$.
 \end{proof}

For $i\in\IZ_3$, consider the subsets $Z_i$ of the dyadic rationals consisting of the numbers whose corresponding binary word has   weight $i$, namely
$$
Z_i:=\{t\in (0,1)\cap \IZ[1/2]\; | \; \tilde\omega(t)= i\} \qquad i\in \IZ_3\; 
$$
\begin{theorem}\label{lemma-inclu}
For any $j\in\IZ_3$, the even part $\CEE$ of the planar $3$-colorable subgroup $\CE$ coincides with 
$\cap_{i\in\IZ_3}{\rm Stab}(Z_i)=\cap_{i\in\IZ_3, i\neq j}{\rm Stab}(Z_i)$.
\end{theorem}
\begin{proof}
First we show that $\CEE=\cap_{i\in\IZ_3}{\rm Stab}(Z_i)$.
The inclusion $\CEE\subset\cap_{i\in\IZ_3}{\rm Stab}(Z_i)$ follows if we   check that the generators of $\CEE$ (namely, 
$\varphi_R(w_0)=\varphi_R(x_0^2x_1x_2^{-1})$,
$\varphi_R(w_1)=\varphi_R(x_0x_1^2x_0^{-1})$,
$\varphi_R(w_2)=\varphi_R(x_1^2x_3x_2^{-1})$,
$\varphi_R(w_3)=\varphi_R(x_2^2x_3x_4^{-1})$,
$x_0^2$)
 preserve the sets $Z_i$ for all $i\in\IZ_3$.
 
 Take $t\in Z_i$ and write it in binary expansion as $t=.a_1\ldots a_n$. 
 The idea of the proof is the following:
 we find 
 tree diagrams representing the generators of $\CEE$ in such a way that each top tree has a leaf whose corresponding binary word is 
 $a_1\ldots a_n$ or $a_1\ldots a_n 0^k$ (for some non-negative integer $k$). This is achieved by adding pairs of opposing carets.
 Because we are considering elements in $\CEE$,
it would then follow that 
  the weight $\tilde\omega(a_1\ldots a_n)=\tilde\omega(a_1\ldots a_n 0^k)$ is equal to 
  the weight of the image of $t$ under them,
  thus 
  the generators of $\CEE$ are also in ${\rm Stab}(Z_i)$.
 
 So   replace each leaf in both the top and bottom trees
 of the generators of $\CEE$
 by a complete binary tree of height $n$ (that is, a tree with $2^n$ leaves, where each leaf has distance $n$ from the root).
 This is done by gluing this tree by its root to the leaf of the top/bottom tree.
 This does not affect the generators, as the 
 new tree diagrams can be reduced to the previous form
by cancelling pairs of opposing carets.  
 Below we exemplify the procedure when $n=2$.
 We draw  a
  leaf from the top tree and a leaf of the bottom tree glued together
  (thus with the two incident edges) and the transformation to be used.
 
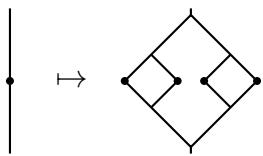
\[
\begin{tikzpicture}[x=.35cm, y=.35cm,
    every edge/.style={
        draw,
      postaction={decorate,
                    decoration={markings}
                   }
        }
]

 \node at (-1.25,-3) {\;};

 \draw[thick] (2,2.75)--(2,-2.75);

 \fill (2,0)  circle[radius=1.5pt];

\end{tikzpicture}
\quad
\begin{tikzpicture}[x=.35cm, y=.35cm,
    every edge/.style={
        draw,
      postaction={decorate,
                    decoration={markings}
                   }
        }
]
 
 \fill (0,0)  circle[radius=1.5pt];
 \fill (2,0)  circle[radius=1.5pt];
 \fill (3,0)  circle[radius=1.5pt];
 \fill (5,0)  circle[radius=1.5pt];

 \node at (-1.25,-3) {\;};

\draw[thick] (0,0) -- (2.5,2.5)--(5,0)--(2.5,-2.5)--(0,0);
 \draw[thick] (1,1) -- (2,0)--(1,-1);
 \draw[thick] (4,1) -- (3,0)--(4,-1);

 \draw[thick] (2.5,2.5)--(2.5,2.75);

 \draw[thick] (2.5,-2.5)--(2.5,-2.75);

\node at (-2,0) {$\mapsto$};
\end{tikzpicture}
 \]

 Now, the matching  leaves of all  new tree diagrams representing the generators of $\CE$ have the same color.
In each of these tree diagrams there is a leaf of the top tree whose corresponding label is $a_1\ldots a_n$ or $a_1\ldots a_n0^k$. 
It follows that $t$ is mapped to another number with the same weight.
 
For the converse inclusion, let $f=(T_+,T_-)$ be an element of $\cap_{i\in\IZ_3}\stab(Z_i)$.
Denote by  $\alpha_+(k)$ and $\alpha_-(k)$ the words associated with the $k$-th leaf of the top and bottom trees, respectively.
Since $f\in \cap_{i\in\IZ_3}\stab(Z_i)$, it follows that $\tilde\omega_+(\alpha_+(k))\equiv_3 \tilde\omega_-(\alpha_-(k))$ for all $k$. 
By Proposition \ref{prop-descr-3col} we have that $f\in\CF$.

Finally we observe that,
for any $j\in\IZ_3$, $\CEE=\cap_{i\in\IZ_3, i\neq j}{\rm Stab}(Z_i)$ because once we fix two colors, then also the third one is fixed.
 \end{proof}
We now recall the definition of closed subgroups of $F$  from \cite[Definition 6.1]{GS4}. Let  $g$ be an element in $F$  which fixes a dyadic rational $\alpha \in (0, 1)$. The components of $g$ at $\alpha$ are the functions
\begin{align*}
&g_1(t)=\left\{\begin{array}{ll}
g(t) & \text{ if } t\in [0,\alpha]\\
t & \text{ if } t\in [\alpha ,1]
\end{array}\right.
\qquad 
g_2(t)=\left\{\begin{array}{ll}
t & \text{ if } t\in [0,\alpha]\\
g(t) & \text{ if } t\in [\alpha ,1]
\end{array}\right.
\end{align*}
A subgroup $H$ of $F$ is said to be closed\footnote{This is not the original definition of closed subgroups, see \cite[Lemma 6.2]{GS4} and the references therein.} if for every function $h \in H$ and every dyadic rational $\alpha$ such that $h$ fixes $\alpha$, the components of $h$ at 
$\alpha$ belong to $H$. 
Closed subgroups have been useful in the study of infinite index maximal subgroups of $F$, see e.g. \cite{GS2, TV2}.
In the recent paper \cite{G2} maximal subgroups of $F$ of infinite index were shown to be closed.

As an immediate consequence of previous theorem we get the following result.
\begin{corollary}\label{CEEclosed}
The even part $\CEE$ of the planar $3$-colorable subgroup   is closed.
\end{corollary} 

\begin{proposition}\label{theo-stab}
The even part $\CEE$ of the planar $3$-colorable subgroup  does not coincide with ${\rm Stab}(Z_i)$ for any $i\in \IZ_3$.
\end{proposition}
\begin{proof}
One can see for example that $x_1^2\in{\rm Stab}(Z_2)\setminus ({\rm Stab}(Z_0)\cup {\rm Stab}(Z_1))$, 
$x_2^2\in{\rm Stab}(Z_1)\setminus ({\rm Stab}(Z_0)\cup {\rm Stab}(Z_2))$, 
$x_0x_1x_2^{-1}x_1^{-1}\in{\rm Stab}(Z_0)\setminus ({\rm Stab}(Z_1)\cup {\rm Stab}(Z_2))$.

We explain how to see this for $x_0x_1x_2^{-1}x_1^{-1}$; for the other elements a similar argument applies.
First we think of the top tree and bottom tree as sitting in the upper-half and lower-half plane, respectively.
We colour the two half planes following Convention \ref{conventionCE},
 so that the leaves of the top tree having color $0$
 meet those of the bottom tree with color $0$.  
We will see soon that  the element is in ${\rm Stab}(Z_0)$, but first we observe that  it is not in ${\rm Stab}(Z_1)\cup {\rm Stab}(Z_2)$ because $x_0x_1x_2^{-1}x_1^{-1}(.01)=.1$, $\tilde\omega(.01)=2$, $\tilde\omega(.1)=1$.
 
The proof of the fact that $x_0x_1x_2^{-1}x_1^{-1}$ is in ${\rm Stab}(Z_0)$ can be done by induction on the length of the binary expansion of the number $t=.a_1\ldots a_n$. 
 Indeed, 
replace each leaf by a complete tree of height $n$
as in the proof of Theorem \ref{lemma-inclu}.
 Now, the leaves of the top tree with weight $0$ meet those of the bottom tree with the same weight (but this is not true for the other weights) and, by construction, there is a leaf in the top tree whose corresponding word is $a_1 \ldots a_n$ (up to addition of zeroes at the end of the word).
It follows that $x_0x_1x_2^{-1}x_1^{-1}$ preserves $Z_0$.
 \end{proof}

It turns out that the subgroup $\CE$ is not a closed subgroup. 
One can, for example, take the   element $h=kx_0=x_1x_2^{-1}x_1^{-1}$, where $k=x_1x_2^{-1}x_1^{-1}x_0^{-1}$ is the element used in the proof of Proposition \ref{actiontransitive}. 
Now $h$ is not in $\CE$ 
(this can be seen for example by checking that the parity of $\log_2 h'$ at $0$ and at $1$ are not the same), but it is one of the two components at $5/8$ of $x_0x_2 x_3^{-1} x_1^{-2}\in \CE$. This means that $\CE$ is not closed.

If one considers Jones' representation   \cite{Jo19} constructed from the planar algebra of quantum $SO(3)$
with $\delta=2\cos(\pi/6)$ (for more information on this planar algebra we refer to \cite{MPS} or to \cite[Section 4]{AJ} for a small summary of the defining relations) one obtains a representation equivalent to
the quasi-regular representation of $F$ associated with $\CE$.  
Jones proved that this representation and those corresponding   to the other admissible values of $\delta$
 are irreducible in \cite{Jo19}. The approach of Jones is completely different to the one in this article. 
 The former is made by exploiting the graphical calculus of both the Thompson group and the planar algebras, and also by finding a special element of $F$ that, in these representations, has a one dimensional eigenspace corresponding to the eigenvalue $1$, whereas ours relies on the description of $\CEE$ in terms of stabilizers of subsets of dyadic rationals. 


In the rest of this section we investigate some quasi-regular representations associated with $\CE$ and $\CEE$.
By means of Theorem \ref{lemma-inclu} 
we will
obtain the following theorem. 
\begin{theorem}\label{theo2}
It holds
\begin{enumerate}
\item Comm$_{F}(\CEE)=\CE$;
\item Comm$_{K_{(2,2)}}(\CEE)=$Comm$_{K_{(2,1)}}(\CEE)=\CEE$.
\end{enumerate} 
\end{theorem}
\begin{corollary}\label{cor-red} 
The quasi-regular representation of $F$ associated with $\CEE$ is reducible. 
\end{corollary}
\begin{proof}[Proof of Corollary \ref{cor-red}]
It follows from 
part 1) of Theorem \ref{theo2}   by a result of Mackey
 \cite{Ma}.
 \end{proof}
 
The following isomorphisms were established in  \cite[Section 3]{TV2}  
\begin{align*}
&\begin{array}{l}
\alpha: \; F\to K_{(2,1)}\\
 x_0\mapsto x_0x_1x_0^{-3}\\
 x_1\mapsto x_0x_1^2x_0^{-3}\\
\end{array}\qquad 
\begin{array}{l}
 \theta: F\to K_{(2,2)}\\
 x_0 \mapsto x_0x_1x_0^{-3}x_1^{-1}\\
 x_1 \mapsto x_0x_1^2x_0^{-3}
\end{array}
\end{align*}
\begin{corollary}\label{cor-irred}
The quasi-regular representations of $F$ associated with $\alpha^{-1}(\CEE)$ and $\theta^{-1}(\CEE)$ are irreducible.
\end{corollary}

The proof of  Theorem  \ref{theo2} is similar to that for the oriented subgroup $\vec{F}$ done by Golan and Sapir    \cite[Theorem 4.15]{GS}. 
As in their proof, we first describe the action of any element of $F$ on a sufficiently
 small neighbourhood of $0$ and then use 
  the description of $\CEE$ in terms of stabilizers of subsets of dyadic rationals to show that $\CEE$ coincides with its commensurator. 
First we present a technical lemma.
\begin{lemma}\label{lemma3}
Let $g\in F$. Then for any $i\in\IZ_3$, there exists $m\in\IN$ 
\begin{enumerate}
\item if $\log_2 f(0)\in2\IZ$,  for any $t\in (0,1/2^m)\cap Z_i$, $\tilde\omega(t)\equiv_3 \tilde\omega(g(t))$;
\item if $\log_2 f(0)\in2\IZ+1$,  for any $t\in (0,1/2^m)\cap Z_0$, 
$g(t)\in Z_0$;
\item if $\log_2 f(0)\in2\IZ+1$,  for any $t\in (0,1/2^m)\cap Z_1$, 
$g(t)\in Z_2$;
\item if $\log_2 f(0)\in2\IZ+1$,  for any $t\in (0,1/2^m)\cap Z_2$, 
$g(t)\in Z_1$.
\end{enumerate}
\end{lemma}
\begin{proof}
By definition, $g(t)=2^l t$ for all $t\in I=[0,2^{-r}]$, where $r\in\IN$, $l\in\IZ$.
If $l\leq 0$, $g(t)$ simply adds $l$ zeros at the beginning of the binary form of $t$. Therefore, we have two cases depending on whether $\log_2 g'(0)$ is in $2\IZ$ or  $2\IZ +1$. 
In the first case it holds $\tilde\omega(t)=\tilde\omega(g(t))$ because of Lemma \ref{lemmaweightsprop}-(3.3), 
while in the second $\tilde\omega(t)\equiv_3 -\tilde\omega(g(t))$  because of Lemma \ref{lemmaweightsprop}-(3.1) and the formula $\omega(.a_1\ldots a_n):= \sum_{i=1}^n (-1)^i a_i$. 
In both cases, it suffices to take $m=r$.

If $l>0$, take $m=\max\{r,l\}+1$. Since $m> r$, the binary word for $t$ begins with at least $l$ zeroes and $g$ erases $l$ of them.  We have two cases as before: $\log_2 g'(0)$ is in $2\IZ$ or  $2\IZ +1$. In the first case $\tilde\omega(t)=\tilde\omega(g(t))$, while in the second $\tilde\omega(t)\equiv_3 -\tilde\omega(g(t))$. 
\end{proof} 

\begin{proof}[Proof of Theorem \ref{theo2}]
First we determine the commensurator of $\CEE$ in $K_{(2,1)}$.
Let $h\in K_{(2,1)}\setminus\CEE$ and set $I:=|\CEE:\CEE\cap h\CEE h^{-1}|$. If $I<\infty$, then there is an $r\in \IN$ such that $x_0^{-r}\in h\CEE h^{-1}$, or equivalently $h^{-1}x_0^{-r}h\in\CEE$ ($x_0$ is one of the generators of $\CEE$).
We will show that for all $n$ big enough, $h^{-1}x_0^{-n}h\not\in\CEE$ (and thus reach a contradiction). 
 By Lemma \ref{lemma3}-(1), there is an $m$ such $\tilde\omega(h(t))\equiv_3\tilde\omega(t)$ for all $t\in [0,2^{-m}]\cap Z_i$, $i\in\IZ_3$.
Since $h\not\in\CEE$, there exists $t\in Z_i$, for some $i\in\IZ_3$, such that $t_1:=h^{-1}(t)\not\in Z_i$. We observe that for all $l\in\IN$, we have $x_0^{-2jl}(t_1)\not\in Z_i$.
There exists an $n\in\IN$ such that $x_0^{-2jn}(t_1)<2^{-m}$. 
Indeed, 
$$
x_0^{-2}(t)=\left\{ 
\begin{array}{ll}
.000\alpha & \text{ if } t=.0\alpha\\
.001\alpha & \text{ if } t=.10\alpha\\
.01\alpha & \text{ if } t=.110\alpha\\
.1\alpha & \text{ if } t=.111\alpha
\end{array}
\right.
$$
We observe that $h(x_0^{-2jn}(t_1))\not\in Z_i$. Then, $h^{-1}x_0^{-2jn}h(t)=h(x_0^{-2jn}(t_1))\not\in Z_i$ and so $h^{-1}x_0^{-2jn}h\not\in {\rm Stab}(Z_i)$.
In particular $h^{-1}x_0^{-2jn}h\not\in \CEE\subset\cap_{i\in\IZ_3}{\rm Stab}(Z_i)$.

 Now we determine the commensurator of $\CEE$ in $F$.
We will show that $x_0$ is in the commensurator of $\CEE$.
We observe that 
$x_0\CEE x_0^{-1}\subset K_{(2,2)}\cap\CE =\CEE$
and
$x_0^{-1}\CEE x_0\subset K_{(2,2)}\cap\CE =\CEE$
 because $x_0\in\CE$.
Therefore, we have 
$\CEE\leq x_0^{-1}\CEE x_0\leq\CEE$ and thus $\CEE=x_0\CEE x_0^{-1}$.
In particular, $x_0$ is in the commensurator of $\CEE$.
Now let $h\in F\setminus \CEE$ such that 
$I:=|\CEE :\CEE\cap h\CEE h^{-1}|<\infty$.
If $h\in K_{(2,1)}$, then by the previous discussion $h\in \CEE$. 
Suppose that $h\in F\setminus K_{(2,1)}$, that is, $\log_2 h'(0)\in 2\IN_0+1$.
If $h\in$ Comm$_F(\CEE)$, then $x_0h\in$Comm$_F(\CEE)$ as well. 
By construction $x_0h\in K_{(2,1)}\cap$Comm$_F(\CEE)=$Comm$_{K_{(2,1)}}(\CEE)=\CEE$. 
It follows that Comm$_F(\CEE)=\langle x_0,\CEE\rangle$.
\end{proof} 

We end this paper with a small summary about Jones' subgroups   and associated representations.
The subgroups $\vec{F}$ and $\CF$ are weakly maximal (i.e., they are maximal among subgroups of infinite index) and self-commensurating, and hence the associated  quasi-regular representations of $F$ are irreducible.
  The planar $3$-colorable subgroup $\CE$ 
   is not weakly maximal (as can be deduced from Golan's computations of its closure)
but it is self-commensurating, hence also the associated  quasi-regular representation of $F$  is irreducible. 
The subgroup $\CEE$ is self-commensurating in the rectangular subgroups $K_{(2,1)}$ and $K_{(2,2)}$, thus the quasi-regular representations of $F$ 
associated with $\alpha^{-1}(\CEE)$ and $\theta^{-1}(\CEE)$
are irreducible. Other examples of maximal and self-commensurating subgroups in $F$  are the stabilizers of points in the unit interval  \cite{Sav}. 
Further examples of weakly maximal subgroups in $F$ were provided in \cite{G, TV2, G2}.

\section*{Acknowledgements}
The authors would like to thank Gili Golan for pointing out to us a mistake in the first version of this paper. 
 V.A. acknowledges the support 
 of the  Mathematisches Institut of the University of  Bern.
 T.N. acknowledges support of  Swiss NSF grants 200020-178828 and 200020-200400.

\end{document}